\documentclass[a4paper]{article}
\usepackage[margin=30mm]{geometry}
\usepackage{amsmath}
\usepackage{amssymb}
\usepackage{amsthm}
\usepackage{braket}
\usepackage{titlesec}
\usepackage{titlefoot}
\usepackage{comment}
\usepackage{bm}
\titleformat*{\section}{\normalsize\bfseries}
\titleformat*{\subsection}{\normalsize\bfseries}
\allowdisplaybreaks
\def\R{\mathbb{R}}
\def\e{{\varepsilon}}        

\def\p{\partial}

\newtheorem{thm}{Theorem}[section]
\newtheorem{lem}[thm]{Lemma}
\newtheorem{cor}[thm]{Corollary}
\newtheorem{prop}[thm]{Proposition}
\newtheorem{rem}[thm]{Remark}

\makeatletter
\@addtoreset{equation}{section}
 
  \makeatother
 
\begin{document}

\title{
\vspace{-1cm}
\large{\bf Asymptotic Profile of Solutions to the Cauchy problem for the Generalized Kadomtsev--Petviashvili Equations \\
with Anisotropic Dissipation in 2D}}
\author{Ikki Fukuda\\ [.7em]
Faculty of Engineering, Shinshu University
}
\date{}
\maketitle


\vspace{-0.75cm}
\begin{abstract}
We consider the Cauchy problem for the generalized Kadomtsev--Petviashvili equations with the dissipation term $-\nu u_{xx}$ in 2D. This is one of the nonlinear dispersive-dissipative type equations, which has a spatial anisotropy. In this paper, we investigate the large time behavior of the solution to this problem. Especially, we show that the $L^{\infty}$-norm of the solution decays at the rate of $t^{-7/4}$ if the initial data $u_{0}(x, y)$ satisfies $(1+|x|)u_{0}\in L^{1}(\R^{2})$ with the zero-mass condition and some appropriate regularity assumptions. Moreover, combining techniques used for parabolic equations and the Schr\"{o}dinger equation, we also derive the detailed asymptotic profile of the solution. 
\end{abstract}


\section{Introduction}  

\indent

We consider the Cauchy problem for the following generalized Kadomtsev--Petviashvili (KP) equations with anisotropic dissipation effect:
\begin{align}\label{KPB-pre}
\begin{split}
& u_{t} + u^{p}u_{x} + u_{xxx} + \e v_{y} -\nu u_{xx} = 0, \ \ v_{x}=u_{y}, \ \ (x, y) \in \R^{2}, \ t>0,\\
& u(x, y, 0) = u_{0}(x, y), \ \ (x, y) \in \R^{2}, 
\end{split}
\end{align}
where $u = u(x, y, t)$ is a real-valued unknown function, $u_{0}(x, y)$ is a given initial data, $p\ge1$ is an integer, $\e \in \{-1, 1\}$ and $\nu>0$. 
The subscripts $x$, $y$ and $t$ denote the partial derivatives with respect to $x$, $y$ and $t$, respectively. 
When $p=1$, \eqref{KPB-pre} is called the KP--Burgers equation which can be regarded as a mathematical model for two-dimensional wave propagation taking into account the effect of the anisotropic dissipation described by $-\nu u_{xx}$. 
We emphasize that the dissipation term $-\nu u_{xx}$ acts only in the main direction of the propagation in \eqref{KPB-pre}, since we are interested in an almost unidirectional propagation. 
The purpose of our study is to investigate the large time behavior of the solutions to \eqref{KPB-pre}. 
In particular, we would like to obtain the optimal $L^{\infty}$-decay estimate of the solutions. 
In addition, we are interested in the detailed asymptotic profile of the solutions. 

First of all, we would like to introduce some known results related to this problem. We note that our target problem \eqref{KPB-pre} is a dissipative version of the following generalized KP equations: 
\begin{align}\label{gKP}
\begin{split}
& u_{t} + u^{p}u_{x} + u_{xxx} + \e v_{y} = 0, \ \ v_{x}=u_{y}, \ \ (x, y) \in \R^{2}, \ t>0,\\
& u(x, y, 0) = u_{0}(x, y), \ \ (x, y) \in \R^{2}.
\end{split}
\end{align}
When $p=1$, \eqref{gKP} is classically called the KP-I ($\e=-1$) or KP-I\hspace{-.1em}I ($\e=1$) equation, which are known as universal models for nearly one directional weakly nonlinear dispersive waves, with weak transverse effects and strong surface tension effects. 
In what follows, let us recall the previous results for \eqref{gKP}. Now, we note that \eqref{gKP} has the following conserved mass and energy:
\begin{equation}\label{E}
M(u)=\int_{\R^2}u^2dxdy,\quad 
E(u)=\int_{\R^2}\left(\frac{1}{2}u_x^2-\frac{\e}{2}v^2-\frac{1}{(p+1)(p+2)}u^{p+2}\right)dxdy.
\end{equation}
Thus, the natural energy space can be defined by
\[
E^1(\R^2):=\left\{u\in L^2(\R^2); \ \|u\|_{L^2}+\|u_x\|_{L^2}+\|v\|_{L^2}<\infty, \ v_x=u_y\right\}. 
\]
Then, Ionescu--Kenig--Tataru~\cite{IKT08} proved the global well-posedness of the KP-I equation in $E^1(\R^2)$. 
On the other hand, Bourgain~\cite{B93} obtained that the KP-I\hspace{-.1em}I equation is globally well-posed in $L^2(\R^2)$. 
Moreover, Hadac--Herr--Koch~\cite{HHK09} showed the global well-posedness of the KP-I\hspace{-.1em}I equation 
for small initial data, in the scaling critical anisotropic Sobolev space $\dot{H}^{-\frac{1}{2},0}(\R^{2})$. 
They also proved that the scattering result of the small solution. We note that the methods used in Bourgain~\cite{B93} and Hadac--Herr--Koch~\cite{HHK09} cannot be applied for the KP-I equation, since the KP-I equation does not have good structure such as the KP-I\hspace{-.1em}I equation in view of the non-resonance. 
Furthermore, the large time behavior of the solution to the KP-I and KP-I\hspace{-.1em}I equations are studied by Hayashi--Naumkin~\cite{HN14}. 
They gave the $L^{\infty}$-decay estimate and the asymptotic formula for $\p_{x}u(x, y, t)$ of the solution to the KP-I or KP-I\hspace{-.1em}I equations, under the smallness assumptions on the initial data. 
Also, the scattering result on $\p_{x}u(x, y, t)$ for the KP-I equation is obtained by Harrop-Griffiths--Ifrim--Tataru~\cite{HIT16}. 
For the generalized KP equation, i.e., \eqref{gKP} with $p\ge2$, Hayashi--Naumkin--Saut~\cite{HNS99} 
obtained the $L^{\infty}$-decay estimates of $u(x, y, t)$ and $\p_{x}u(x, y, t)$. Moreover, they derived the asymptotic formula for $\p_{x}u(x, y, t)$ of the small solution to \eqref{gKP}. After that, the results on the decay estimates given in \cite{HNS99} were improved by Niizato~\cite{N11}. 
We remark that these results given in \cite{HN14, HNS99, N11}, the decaying and the regularity of $\partial_x^{-1}u_0$ are assumed for the initial data $u_0(x, y)$, where $\partial_x^{-1}$ denotes the anti-derivative operator which will be defined by \eqref{anti-derivative} below. 
However, the assumptions for $\partial_x^{-1}u_0$ are not unnatural because of the form of the corresponding energy. 

Although we have reviewed results on \eqref{gKP}, our present analysis is essentially different in nature from those. 
Actually, compared with \eqref{gKP}, our target problem \eqref{KPB-pre}, which has the dissipation term $-\nu u_{xx}$, 
should be considered as a dispersive-dissipative type equation, not as a purely dispersive equation. 
Therefore, to analyze this problem, it is necessary to employ not only methods developed for dispersive equations but also techniques for parabolic equations. 
For such equations, there are relatively few results in higher dimensions, and most results are concerned with the one-dimensional case. A typical example of such equations is the following generalized KdV--Burgers equation: 
\begin{align}\label{KdVB}
\begin{split}
& u_{t} + u^{p}u_{x} + u_{xxx} -\nu u_{xx} = 0, \ \ x \in \R, \ t>0,\\
& u(x, 0) = u_{0}(x), \ \ x \in \R, 
\end{split}
\end{align}
which is a one-dimensional version of \eqref{KPB-pre}. Since the global well-posedness of \eqref{KdVB} in the Sobolev spaces can be easily proved by virtue of the dissipation effect, the main topic is about the asymptotic analysis. The first study of the KdV--Burgers equation, i.e., \eqref{KdVB} with $p=1$, is given by Amick--Bona--Schonbek~\cite{ABS89}. They derived the $L^{q}$-decay estimates which have the same decay rates of the solution to the heat equation. It means that the dissipation effect $-\nu u_{xx}$ is stronger than the effects of the dispersion $u_{xxx}$. 
Moreover, Karch~\cite{K99-2} extended the results given in \cite{ABS89} and obtained the asymptotic formula for the solution to \eqref{KdVB} with $p=1$. 
He showed that if the initial data satisfies $u_{0}\in L^{1}(\R)$ and some appropriate regularity assumptions, then the solution $u(x, t)$ converges to the self-similar solution to the Burgers equation: 
\begin{equation*}
\chi_{t}+\chi \chi_{x}-\nu \chi_{xx}=0, \ \ x\in \R, \ t>0.  
\end{equation*}
Hayashi--Naumkin~\cite{HN06} improved their result given in \cite{K99-2}. 
Moreover, Kaikina--Ruiz-Paredes~\cite{KP05} derived the second asymptotic profile of the solution, under the additional assumption $xu_{0}\in L^{1}(\R)$. 
In view of the second asymptotic profile, they found that the effect of the dispersion term $u_{xxx}$ appears from the second term of its asymptotics. 
For a generalization of these results, see also \cite{F19-1}. On the other hand, for the generalized KdV--Burgers equation, i.e., \eqref{KdVB} with $p\ge2$, Karch~\cite{K99-1} showed that the solution to $u(x, t)$ tends to the heat kernel instead of the self-similar solution to the Burgers equation. In addition, he also derived the second term of asymptotics for the solution, in the same paper \cite{K99-1}. For some related results on this topic, let us refer to \cite{HKNS06}. 

The above one-dimensional results tell us that the dispersion term $u_{xxx}$ can be negligible in the sense of the first approximation, due to the dissipation effect $-\nu u_{xx}$. From this point of view, we can expect that the same phenomenon also happen in our target equation \eqref{KPB-pre}. 
The author and Hirayama \cite{FH23} studied \eqref{KPB-pre} to demonstrate this issue. In what follows, we would like to explain the known results for \eqref{KPB-pre} related to the present study. In particular, we focus on the results given in \cite{M99} and \cite{FH23}. 
In order to analyze \eqref{KPB-pre}, we shall rewrite \eqref{KPB-pre} in the form of an evolution equation. 
Based on Molinet~\cite{M99}, for $m \in \mathbb{N}$ and $s\ge0$, let us introduce the following function spaces: 
\begin{align}
& \dot{\mathbb{H}}^{-m}_{x}(\R^{2}) := \left\{ f \in \mathcal{S}'(\R^{2}); \ \|f\|_{\dot{\mathbb{H}}^{-m}_{x}} := \left\| \xi^{-m} \hat{f} \right\|_{L^{2}} < \infty \right\},  \label{space-Hm} \\
& X^{s}(\R^{2}) := \left\{ f \in H^{s}(\R^{2}); \ \|f\|_{X^{s}} := \|f\|_{H^{s}} + \left\| \mathcal{F}^{-1}\left[ \xi^{-1}\hat{f}(\xi, \eta)\right] \right\|_{H^{s}} <\infty \right\}.  \label{space-Xs}
\end{align}
Moreover, we shall define the anti-derivative operator $\p_{x}^{-1}$ acting on $\dot{\mathbb{H}}^{-1}_{x}(\R^{2})$ by 
\begin{equation}\label{anti-derivative}
\p_{x}^{-1}f(x, y):=\mathcal{F}^{-1}\left[ (i\xi)^{-1}\hat{f}(\xi, \eta) \right](x, y).
\end{equation}
The definition of the Fourier transform $\hat{f}(\xi, \eta)=\mathcal{F}[f](\xi, \eta)$ will be given at the end of this section. 
We can easily see that $X^{s}(\R^{2})$ is embedded in $\dot{\mathbb{H}}^{-m}_{x}(\R^{2})$ 
for any $m\in \mathbb{N}$ and $s\ge 0$. 
Therefore, for $u \in X^{1}(\R^{2})$, $v_{x}=u_{y}$ in \eqref{KPB-pre} can be rewritten by $v = \p_{x}^{-1}u_{y}$. 
Because of this, \eqref{KPB-pre} is equivalent to the following Cauchy problem: 
\begin{align}\label{KPB}
\begin{split}
& u_{t} + u^{p}u_{x} + u_{xxx} + \e \p_{x}^{-1}u_{yy} -\nu u_{xx} = 0, \ \ (x, y) \in \R^{2}, \ t>0,\\
& u(x, y, 0) = u_{0}(x, y), \ \ (x, y) \in \R^{2}.  
\end{split}
\end{align}

In what follows, we would like to consider \eqref{KPB} instead of \eqref{KPB-pre}. 
First, let us introduce the fundamental results about the global well-posedness of \eqref{KPB}. Molinet~\cite{M99} constructed the solutions to \eqref{KPB}, provided the initial data $u_{0} \in X^{s}(\R^{2})$ for $s>2$. More precisely, he showed that if $p\ge1$ and $B(u_{0})$ is sufficiently small, then \eqref{KPB} has a unique global mild solution $u(x, y, t)$ satisfying 
\begin{align*}
&u \in C_{b}([0, \infty); X^{s}(\R^{2})) \cap L^{2}(0, \infty; H^{s}(\R^{2})), \\
&\p_{x}^{l}u \in C_{b}([1, \infty); H^{s}(\R^{2})) \cap L^{2}(1, \infty; H^{s}(\R^{2})), \ \ l \in \mathbb{N}, 
\end{align*}
and the mapping $u_{0} \mapsto u$ is continuous from $X^{s}(\R^{2})$ into $C([0, \infty); X^{s}(\R^{2}))$, where 
\begin{align*}
B(u_{0}):=\left\|u_{0}\right\|_{H^{1}}^{2}+\left\|\p_{x}^{2}u_{0}\right\|_{L^{2}}^{2}+\left\|\p_{x}^{-1}\p_{y}u_{0}\right\|_{L^{2}}^{2}. 
\end{align*}
This quantity $B(u_{0})$ arises from the structure of the energy \eqref{E}. Moreover, the solution satisfies the following a priori estimate: 
\begin{equation}\label{apriori-Xs}
\left\|u(t)\right\|_{X^{s}}^{2}+\int_{0}^{t}\left\|u_{x}(\tau)\right\|_{X^{s}}^{2}d\tau \le C_{\dag}\left(B(u_{0}), \|u_{0}\|_{X^{s}}\right), \ \ t\ge 0,  
\end{equation}
where $C_{\dag}\left(B(u_{0}), \|u_{0}\|_{X^{s}}\right)$ is a certain positive constant depending on $B(u_{0})$ and $\|u_{0}\|_{X^{s}}$, which goes to zero as $u_{0} \to 0 $ in $X^{s}(\R^{2})$. For the related results about the global well-posedness of \eqref{KPB}, we can also refer to another his paper \cite{M00}. 
Furthermore, for the case $p=1$, we also remark that the global well-posedness of \eqref{KPB} is improved by Kojok~\cite{K07} for KP-I\hspace{-.1em}I and Darwich~\cite{D12} for KP-I (see, also \cite{MR02}). 

In addition to the global well-posedness, Molinet also discussed the large time behavior of the solution to \eqref{KPB} in \cite{M99}. More precisely, he gave some results about the $L^{\infty}$-decay estimates for the solution. In what follows, let $s=3$ for simplicity. 
Then, under the assumptions $u_{0} \in X^{3}(\R^{2})$, $\p_{x}^{-1}u_{0} \in L^{1}(\R^{2})$ and $\p_{x}^{-1}\p_{y}^{3}u_{0} \in L^{1}(\R^{2})$, the solution to \eqref{KPB} satisfies 
\begin{equation}\label{u-decay-Molinet}
\left\|u(t)\right\|_{L^{\infty}} \le C
\begin{cases}
t^{-\frac{7}{4}},& p\ge2,  \\
t^{-\frac{3}{2}},& p=1,  
\end{cases}
\end{equation}
for all $t\ge1$. In fact, the estimates for the higher-order derivatives have also been established, and it is known that the decay rate improves by $1/2$ with each $x$-derivative (for simplicity, we omit the details here). 
We note that this decay estimate \eqref{u-decay-Molinet} is a different from not only the one for the solution to \eqref{KdVB} but also the one for the solution to two-dimensional parabolic equations and dispersive equations. 
Actually, this estimate \eqref{u-decay-Molinet} is constructed by the interactions of the effects of the dissipation $-\nu u_{xx}$ in the $x$-direction, the dispersion $\e \p_{x}^{-1}u_{yy}$ in both the $x$-direction and the $y$-direction, and also the anti-derivative $\p_{x}^{-1}$ on the initial data $u_{0}(x, y)$. 

Recently, the author and Hirayama \cite{FH23} investigated the optimality for the decay rate $t^{-7/4}$ given in \eqref{u-decay-Molinet} if $p\ge2$, by deriving the lower bound of the $L^{\infty}$-norm for the solution $u(x, y, t)$ to \eqref{KPB}. More precisely, under the same assumptions in the results given by Molinet~\cite{M99} and $p\ge2$, there exists a positive constant $C_{\dag}>0$ such that the solution $u(x, y, t)$ to \eqref{KPB} satisfies 
\begin{equation}\label{u-est-lower-FH}
\liminf_{t\to \infty}t^{\frac{7}{4}}\left\|u(t)\right\|_{L^{\infty}} \ge C_{\dag}\left|\mathcal{M}_{0}\right|,  
\end{equation}
where the real constant $\mathcal{M}_{0}$ is defined by 
\begin{equation}\label{DEF-mathM}
\mathcal{M}_{0}:=\int_{\R^{2}}\p_{x}^{-1}u_{0}(x, y)dxdy-\frac{1}{p+1}\int_{0}^{\infty}\int_{\R^{2}}u^{p+1}(x, y, t)dxdydt. 
\end{equation}
By virtue of the above result, for $p\ge2$, we can conclude that the $L^{\infty}$-decay estimate \eqref{u-decay-Molinet} is optimal with respect to the decay rate $t^{-7/4}$, under the condition $\mathcal{M}_{0}\neq0$. In addition, they also give a sufficient condition for $\mathcal{M}_{0}\neq0$ in \cite{FH23}. Furthermore, they derived that an approximation formula for the solution to \eqref{KPB}. Now, let us consider the following problem: 
\begin{align}
\begin{split}\label{phi-eq}
& \phi_{t} + \e \p_{x}^{-1}\phi_{yy} -\nu \phi_{xx} = -u^{p}u_{x}, \ \ (x, y) \in \R^{2}, \ t>0,\\
& \phi(x, y, 0) = u_{0}(x, y), \ \ (x, y) \in \R^{2}, 
\end{split}
\end{align}
where $u(x, y, t)$ is the original solution to \eqref{KPB}. 
Then, under the same assumptions to get \eqref{u-decay-Molinet} and \eqref{u-est-lower-FH}, we can see that the solution $u(x, y, t)$ to \eqref{KPB} is well approximated by the above function $\phi(x, y, t)$. More precisely, the following approximation formula has been established: 
 \begin{equation}
\left\|u(t)-\phi(t)\right\|_{L^{\infty}}\le Ct^{-\frac{9}{4}}, \ \ t\ge2, \label{u-approximation}
\end{equation}
This approximation formula \eqref{u-approximation} tells us that the dispersion term $u_{xxx}$ in the direction in which the dissipation term $-\nu u_{xx}$ is working can be negligible as $t\to \infty$. Therefore, we can say that \eqref{KPB} has the similar property of the one-dimensional problem \eqref{KdVB}. 

As we mentioned in the above, a lot of results have already been obtained for \eqref{KPB} such as the decay estimates of the solution and the approximation formula. However, several issues remain to be discussed. First, regarding the optimality for the decay rate of the solution, only the case $p\ge2$ has been discussed in \cite{FH23}, and it is not known whether the decay rate $t^{-3/2}$ obtained in the estimate (1.9) is optimal in the case $p=1$. Another issue is that the approximation formula \eqref{u-approximation} does not provide detailed information on the profile of the solution. Actually, the equation \eqref{phi-eq}, which determines $\phi(x, y, t)$ contains the original solution $u(x, y, t)$ itself in its inhomogeneous term. Thus, neither an explicit representation of $\phi(x, y, t)$ nor its properties can be obtained. From a practical point of view, one should derive an asymptotic profile that can be expressed explicitly. 
Furthermore, some of the assumptions on the initial data seem to be a little bit restrictive. Indeed, $u_{0}(x, y)$ are assumed to satisfy $\p_{x}^{-1}u_{0} \in L^{1}(\R^{2})$ and $\p_{x}^{-1}\p_{y}^{3}u_{0} \in L^{1}(\R^{2})$. However, these conditions do not appear to be natural. At least, it is desirable to perform the asymptotic analysis without assuming $\p_{x}^{-1}\p_{y}^{3}u_{0} \in L^{1}(\R^{2})$. Also, we would like to analyze \eqref{KPB} under the condition $u_{0}\in L^{1}(\R^{2})$ instead of $\p_{x}^{-1}u_{0} \in L^{1}(\R^{2})$. These are the motivations for our study. 

Finally, before stating our main results, we note that an additional constraint which is called the zero-mass condition are implicitly imposed on the initial data $u_{0}(x, y)$, under the current framework. More precisely, if $s>2$ and $u_{0}\in X^{s}(\R^{2})\cap L^{1}(\R^{2})$, then we have 
\begin{equation}\label{zero-mass}
\int_{\R}u_{0}(x, y)dx=0, \ \ {\rm a.e.} \ y\in \R. 
\end{equation}
Actually, if we set $v_{1}(x, y)=\p_{x}^{-1}u_{0}(x, y)$ and $v_{2}(x, y):=\int_{-\infty}^{x}u_{0}(z, y)dz$, 
then $\hat{v}_{1}(\xi, y)=\hat{v}_{2}(\xi, y)$ in $\mathcal{S}'(\R^{2})$. 
Thus, we have $v_{1}=v_{2}$, i.e., the anti-derivative operator $\p_{x}^{-1}$ can be rewritten by 
\begin{equation}\label{anti-int}
\p_{x}^{-1}u_{0}(x, y)=\int_{-\infty}^{x}u_{0}(z, y)dz. 
\end{equation}
Now, we note that $v_{1}=\p_{x}^{-1}u_{0}\in H^{s}(\R^{2})$ because $u_{0}\in X^{s}(\R^{2})$. 
Therefore, it follows from $s>2$ and the property of the Sobolev space that 
\[
0=\lim_{x\to \infty}v_{1}(x, y)=\lim_{x\to \infty}v_{2}(x, y)=\lim_{x\to \infty}\int_{-\infty}^{x}u_{0}(z, y)dz=\int_{\R}u_{0}(z, y)dz, \ \ {\rm a.e.} \ y\in \R. 
\]
It means \eqref{zero-mass} holds. Namely, we have $\mathcal{F}_{x}[u_{0}](0, y)=0$. 
This avoids the singularity at $\xi=0$ that appears in \eqref{anti-derivative} for the definition of $\p_{x}^{-1}$. 
For further discussion related on this matter, see e.g. Molinet--Saut--Tzvetkov~\cite{MST07}, Mammeri~\cite{Ma09} and also references therein. 
Also, this condition plays an essential role throughout the analysis of this paper. 

\bigskip
\par\noindent
\textbf{\bf{Main results.}} 

\medskip
In what follows, we state our main results. First, for simplicity, we set 
\begin{align}\label{data-1}
E_{0}:=\sqrt{B(u_{0})}+\left\|u_{0}\right\|_{L^{1}}+\left\|xu_{0}\right\|_{L^{1}}+\left\|\p_{x}^{-1}u_{0}\right\|_{L^{2}}^{2}. 
\end{align}
Next, in order to state the asymptotic formula, let us introduce the following functions: 
\begin{align}
\begin{split}\label{DEF-K}
& K(x, y, t) := t^{ -\frac{5}{4} } K_{*} \left( xt^{-\frac{1}{2}}, yt^{-\frac{3}{4}} \right), \ \ (x, y)\in \R^{2}, \ t>0, \\
& K_{*}(x, y) := \frac{ 1 }{ 4\pi^{ \frac{3}{2} } \nu^{\frac{3}{4}} } \int_{0}^{\infty} r^{-\frac{1}{4} }e^{ -r } \cos \left( x \sqrt{\frac{r}{\nu}} +\frac{y^{2}}{4\e}\sqrt{\frac{r}{\nu}} - \frac{\pi}{4}\e \right) dr, \ \ (x, y)\in \R^{2}.  
\end{split}
\end{align}
We note that this $K(x, y, t)$ is the fundamental solution to the following equation: 
\[
\psi_{t}+ \e \p_{x}^{-1}\psi_{yy} -\nu \psi_{xx} = 0, \ \ (x, y) \in \R^{2}, \ t>0 
\]
(cf.~\cite{FH23}). In addition, we also define the real constant $\mathcal{N}_{0}$ as follows: 
\begin{equation}\label{DEF-mathN}
\mathcal{N}_{0}:=\int_{\R^{2}}(-x)u_{0}(x, y)dxdy-\frac{1}{p+1}\int_{0}^{\infty}\int_{\R^{2}}u^{p+1}(x, y, t)dxdydt. 
\end{equation}
Then, we have the following anisotropic asymptotic formula for the solution to \eqref{KPB}: 
\begin{thm}\label{thm.main}
Let $p\ge1$ be an integer. Assume that $u_{0}\in X^{3}(\R^{2})\cap L^{1}(\R^{2})$, $xu_{0}\in L^{1}(\R^{2})$ and $B(u_{0})$ is sufficiently small. 
Then, the solution $u(x, y, t)$ to \eqref{KPB} satisfies the following estimate: 
\begin{equation}\label{u-sol-decay}
\left\|u(t)\right\|_{L^{\infty}} \le CE_{0}(1+t)^{-\frac{7}{4}}, \ \ t\ge0.  
\end{equation}
Moreover, there exist a bounded function $\omega : [0, \infty)\to [0, \infty)$ satisfying $\lim_{\rho \to 0+}\omega(\rho)=0$, and a remainder function $R(x, y, t)$ such that 
\begin{equation}\label{u-sol-asymp}
\left|u(x, y, t)-\mathcal{N}_{0}\p_{x}K(x, y, t)\right| \le CE_{0}^{p}y^{2}t^{-\frac{13}{4}}+t^{-\frac{7}{4}}\omega\left(|y|t^{-\frac{3}{2}}\right)+R(x, y, t),  
\end{equation}
for each $(x, y)\in \R^{2}$, $t>0$, and 
\begin{equation}\label{Rr-est}
\lim_{t\to \infty}t^{\frac{7}{4}}\left\|R(t)\right\|_{L^{\infty}}=0, 
\end{equation}
where $K(x, y, t)$, $\mathcal{N}_{0}$ and $E_{0}$ are defined by \eqref{DEF-K}, \eqref{DEF-mathN} and \eqref{data-1}, respectively. 
In particular, for any positive function $r(t)=o(t^{3/4})$ $(t\to \infty)$, we have the following asymptotic formula:  
\begin{equation}\label{u-sol-asymp-2}
\lim_{t\to \infty}t^{\frac{7}{4}}\sup_{\substack{x\in\mathbb{R}, \\ |y|\leq r(t)}}\left|u(x, y, t)-\mathcal{N}_{0}\p_{x}K(x, y, t)\right|=0. 
\end{equation}
\end{thm}

\begin{rem}
{\rm 
Under the assumptions of Theorem~\ref{thm.main}, $u_{0}(x, y)$ satisfies $(1+|x|)u_{0} \in L^{1}(\R^{2})$ and the zero-mass condition \eqref{zero-mass}. In this case, from some basic calculations based on \eqref{anti-int} and the integration by parts, we can easily get $\p_{x}^{-1}u_{0}\in L^{1}(\R^{2})$. Moreover, the following relation holds:  
\[
\int_{\R}\p_{x}^{-1}u_{0}(x, y)dx=\int_{\R}(-x)u_{0}(x, y)dx, \ \ {\rm a.e.} \ y\in \R. 
\]
Thus, we can see that $\mathcal{M}_{0}$ defined by \eqref{DEF-mathM} and $\mathcal{N}_{0}$ defined by \eqref{DEF-mathN} satisfy $\mathcal{M}_{0}=\mathcal{N}_{0}$. 
In other words, $(-x)u_{0}(x, y)$ is essentially the same as $\p_{x}^{-1}u_{0}(x, y)$. 
However, we remark that there is no inclusion relation between our assumptions and the assumptions in \cite{M99} and \cite{FH23} mentioned above, because they assumed that $u_{0}(x, y)$ satisfies not only $\p_{x}^{-1}u_{0}\in L^{1}(\R^{2})$ but also $\p_{x}^{-1}\p_{y}^{3}u_{0}\in L^{1}(\R^{2})$. 
}
\end{rem}

\begin{rem}
{\rm 
We emphasize that the asymptotic profile of the solution $u(x, y, t)$ to \eqref{KPB} cannot be described by a purely linear approximation alone. Although the leading term is expressed through a linear kernel $\p_{x}K(x, y, t)$,
its amplitude $\mathcal{N}_{0}$ is determined by the nonlinear evolution of the solution. 
}
\end{rem}

\begin{rem}
{\rm 
We note that the decay rate of the leading term of the solution $u(x,y,t)$ is $t^{-7/4}$, since $\|\partial_x K(t)\|_{L^\infty}=\|\partial_x K_*\|_{L^\infty}t^{-7/4}$ (also recall the uniform estimate (1.20)). Therefore, the first error term on the right hand side of \eqref{u-sol-asymp} is of lower order than the leading term if $|y|t^{-3/4}=o(1)$. On the other hand, since $\omega(\rho)\to 0$ as $\rho\to 0+$, the second error term $t^{-7/4}\omega(|y|t^{-3/2})$ is of lower order than $t^{-7/4}$ if $|y|t^{-3/2}=o(1)$. Thus, although two different anisotropic scales $t^{3/4}$ and $t^{3/2}$ appear in the error estimate \eqref{u-sol-asymp}, the former gives the stronger restriction. Consequently, the asymptotic profile $\mathcal{N}_{0}\partial_{x} K(x, y, t)$ describes the leading behavior of the solution in the region $|y|=o(t^{3/4})$, as stated in \eqref{u-sol-asymp-2}. In contrast, a different asymptotic description is required in the region where $|y|t^{-3/4}$ does not tend to zero. Such a phenomenon has also been observed for the generalized KP equations \eqref{gKP}. For details, see Theorem~1.2 in \cite{HNS99} and Corollary~1.2 in \cite{HN14}.
}
\end{rem}

In addition to the above results, it follows from the triangle inequality and \eqref{DEF-K} that 
\begin{align*}
\left\|u(t)\right\|_{L^{\infty}}
&\ge |u(0, 0, t)| \ge \left|\mathcal{N}_{0}(\p_{x}K)(0, 0, t)\right|-\left|u(0, 0, t)-\mathcal{N}_{0}\p_{x}K(0, 0, t)\right| \\
&= t^{-\frac{7}{4}}\left|(\p_{x}K_{*})(0, 0)\right|\left|\mathcal{N}_{0}\right|-\left|u(0, 0, t)-\mathcal{N}_{0}\p_{x}K(0, 0, t)\right|.  
\end{align*}
Therefore, taking the limit $t\to \infty$ in the above and using the asymptotic formula \eqref{u-sol-asymp} with \eqref{Rr-est}, we can obtain the following lower bound of the $L^{\infty}$-norm of the solution $u(x, y, t)$: 
\begin{cor}\label{cor.u-est-lower}
Under the assumptions in Theorem~\ref{thm.main}, the solution $u(x, y, t)$ to \eqref{KPB} satisfies  
\begin{equation*}
\liminf_{t\to \infty}t^{\frac{7}{4}}\left\|u(t)\right\|_{L^{\infty}} \ge \frac{\Gamma(\frac{5}{4})}{2^{\frac{5}{2}}\pi^{\frac{3}{2}}\nu^{\frac{5}{4}}}\left|\mathcal{N}_{0}\right|. 
\end{equation*}
\end{cor}

\begin{rem}
{\rm 
By virtue of this result, we can conclude that the $L^{\infty}$-decay estimate \eqref{u-sol-decay} is optimal with respect to the decay rate $t^{-7/4}$, if $\mathcal{N}_{0}\neq0$ is satisfied. A sufficient condition for $\mathcal{N}_{0}\neq0$ is discussed in \cite{FH23}, and the similar condition also can be derived under the current assumptions. Roughly speaking, an additional smallness assumption on the initial data is a key to get such a condition. For details, see Remarks~3.8 and 3.9 in \cite{FH23}. 
On the other hand, if $p\in \mathbb{N}$ is odd, $\mathcal{N}_{0}\neq0$ can be realized without imposing the additional smallness on the initial data, except that $B(u_{0})$ is small. 
Actually, in this case, the second term in the right hand side of \eqref{DEF-mathN} is always negative. 
Therefore, by choosing $u_{0}(x, y)$ such that $\int_{\R^{2}}xu_{0}(x, y)dxdy\ge0$, then $\mathcal{N}_{0}<0$ is guaranteed. 
A typical example of such an initial data is $u_{0}(x, y)=cxe^{-(x^{2}+y^{2})}$, where $c>0$ is a small constant. 
}
\end{rem}

The rest of this paper is organized as follows. In Section 2, we prepare several preliminary results. In particular, we establish the detailed asymptotic expansions for the linearized equation and derive several key estimates. In Section 3, we prove the main results of this paper. We first establish the optimal decay estimate of the solution. After that, we derive the explicit asymptotic profile of the solution by analyzing the integral equation \eqref{integral-eq} below. 
The main novelty of this paper is the derivation of an explicit asymptotic profile whose structure can be described in a precise and explicit form.
In particular, we show that the leading term is given by the derivative of the fundamental solution to \eqref{w-linear} below and that its amplitude is determined by a nonlinear quantity associated with the initial data $u_{0}(x, y)$. Moreover, our analysis removes some restrictive assumptions imposed in previous works \cite{M99} and \cite{FH23}, and provides a more complete description of the large time behavior of the solution $u(x, y, t)$ to \eqref{KPB}.

\bigskip
\par\noindent
\textbf{\bf{Notations.}} 

\medskip
We define the Fourier transform of $f$ and the inverse Fourier transform of $g$ as follows:
\[
\hat{f}(\xi, \eta) = \mathcal{F}[f](\xi, \eta) := \frac{1}{2\pi}\int_{\R^{2}}e^{-ix\xi-iy\eta}f(x, y)dxdy, \ \ 
\mathcal{F}^{-1}[g](x, y):=\frac{1}{2\pi}\int_{\R^{2}}e^{ix\xi+iy\eta}g(\xi, \eta)d\xi d\eta.
\]
We also define the anisotropic Fourier transforms $\mathcal{F}_x$, $\mathcal{F}^{-1}_{\xi}$, $\mathcal{F}_y$ and $\mathcal{F}^{-1}_{\eta}$ as follows:
\begin{align*}
&\mathcal{F}_x[f](\xi, y) := \frac{1}{\sqrt{2\pi}}\int_{\R}e^{-ix\xi}f(x, y)dx, \ \ 
\mathcal{F}_{\xi}^{-1}[g](x, y):=\frac{1}{\sqrt{2\pi}}\int_{\R}e^{ix\xi}g(\xi, y)d\xi, \\
&\mathcal{F}_y[f](x, \eta) := \frac{1}{\sqrt{2\pi}}\int_{\R}e^{-iy\eta}f(x, y)dy, \ \ 
\mathcal{F}_{\eta}^{-1}[g](x, y):=\frac{1}{\sqrt{2\pi}}\int_{\R}e^{iy\eta}g(x, \eta)d\eta.
\end{align*}

For $1\le p\le \infty$, $L^{p}(\R^{2})$ means the Lebesgue spaces. 
Also, the Schwartz space and its dual space are denoted by $\mathcal{S}(\R^{2})$ and $\mathcal{S}'(\R^{2})$, respectively. 
Moreover, for $s\in \R$, $H^{s}(\R^{2})$ denotes the Sobolev spaces. In addition, for $s_{1}\in \R$ and $s_{2} \in \R$, we use the following anisotropic Sobolev spaces: 
\begin{align*}
H^{s_{1}, s_{2}}(\R^{2}):=\left\{f \in \mathcal{S}'(\R^{2}); \ \left\|f\right\|_{H^{s_{1}, s_{2}}}:=\left\|\left(1+\xi^{2}\right)^{\frac{s_{1}}{2}}\left(1+\eta^{2}\right)^{\frac{s_{2}}{2}}\hat{f}(\xi, \eta)\right\|_{L^{2}_{\xi \eta}}<\infty \right\}. 
\end{align*}
Furthermore, we also use the additional function spaces $\dot{\mathbb{H}}^{-m}_{x}(\R^{2})$ and $X^{s}(\R^{2})$ being defined by \eqref{space-Hm} and \eqref{space-Xs}, respectively. The anti-derivative operator $\p_{x}^{-1}$ is defined by \eqref{anti-derivative}.

\medskip
Let $T>0$, $1\le q \le \infty$ and $X$ be a Banach space. Then, $L^{q}(0,T; X)$ denotes the space of all measurable functions $u: (0, T) \to X$ such that $\|u(t)\|_{X}$ belongs to $L^{q}(0, T)$. Also, $C([0, T]; X)$ denotes the subspace of $L^{\infty}(0, T; X)$ of all continuous functions $u: [0, T] \to X$. Moreover, $C_{b}([0, \infty); X)$ is defined as the space of all continuous and uniformly bounded functions $u: [0, \infty) \to X$. 

\medskip
Throughout this paper, $C$ denotes various positive constants, which may vary from line to line during computations. Also, it may depend on the norm of the initial data. However, we note that it does not depend on the space variable $(x, y)$ and the time variable $t$. 

\section{Preliminaries}  

\indent

In this section, we would like to introduce some preliminary results for the analysis of \eqref{KPB}. 
Most of the results given in this section are generalizations or refinements of those obtained in \cite{M99} and \cite{FH23}.
First of all, in order to derive the decay estimate and the asymptotic profile of the solution, let us rewrite \eqref{KPB} into the corresponding integral equation. 
Throughout this paper, all analyses related to the main results are based on the following equation: 
\begin{equation}\label{integral-eq}
u(t) = S(t)*u_{0} -\frac{1}{p+1}\int_{0}^{t} \p_{x}S(t-\tau)*u^{p+1}(\tau)d\tau, 
\end{equation}
where the integral kernel $S(x, y, t)$ is defined by 
\begin{equation}\label{DEF-S}
S(x, y, t) := \frac{1}{2\pi}\mathcal{F}^{-1} \left[ e^{ - \nu t \xi^{2} + it\left(\xi^{3} - \e \frac{\eta^{2}}{\xi} \right) } \right] (x, y), \ \ (x, y)\in \R^{2}, \ t>0. 
\end{equation}

\subsection{Linear Analysis}  

\indent

First, let us analyze the linear part of \eqref{integral-eq}. 
We start with introducing the following $L^{\infty}$-decay estimates for the integral kernel $S(x, y, t)$, given in \cite{M99}. 
For the proofs, see Lemma~4.3 in \cite{M99}.
\begin{prop}\label{prop.L-decay-S}
Let $l$ be a non-negative integer. Then, we have 
\begin{equation*}
\left\| \p_{x}^{l} S(t) \right\|_{L^{\infty}} \le C t^{ -\frac{5}{4} -\frac{l}{2} }, \ \ t>0. 
\end{equation*}
\end{prop}
\begin{cor}\label{cor.L-decay-LKPB}
Let $l$ be a non-negative integer. If $u_{0}\in L^{1}(\R^{2})$, then we have  
\begin{equation*}
\left\| \p_{x}^{l}S(t)*u_{0}\right\|_{L^{\infty}} \le C\left\|u_{0}\right\|_{L^{1}}t^{-\frac{5}{4}-\frac{l}{2}}, \ \ t>0. 
\end{equation*}
\end{cor}
Next, we shall construct the asymptotic profile for the integral kernel $S(x, y, t)$. In order to state such a result, let us consider the following auxiliary Cauchy problem:
\begin{align}\label{w-linear}
\begin{split}
& \psi_{t} + \e \p_{x}^{-1}\psi_{yy} -\nu \psi_{xx} = 0, \ \ (x, y) \in \R^{2}, \ t>0,\\
& \psi(x, y, 0) = u_{0}(x, y), \ \ (x, y) \in \R^{2}. 
\end{split}
\end{align}
By using the integral kernel $K(x, y, t)$ defined by \eqref{DEF-K} introduced in the first section, the solution to this Cauchy problem can be written as follows (for details, see \cite{FH23}):
\begin{equation}\label{w-sol}
\psi(x, y, t)=(K(t)*u_{0})(x, y), \ \ (x, y) \in \R^{2}, \ t>0. 
\end{equation}
Moreover, for the integral kernel $K(x, y, t)$ and the solution $\psi(x, y, t)$ to \eqref{w-linear}, we can get the following $L^{\infty}$-decay estimates similar to Proposition~\ref{prop.L-decay-S} and Corollary~\ref{cor.L-decay-LKPB}: 
\begin{prop}\label{prop.L-decay-K}
Let $l$ be a non-negative integer. Then, we have 
\begin{equation*}
\left\| \p_{x}^{l} K(t) \right\|_{L^{\infty}} = \left\|\p_{x}^{l}K_{*}\right\|_{L^{\infty}}t^{ -\frac{5}{4} -\frac{l}{2} }, \ \ t>0. 
\end{equation*}
\end{prop}
\begin{cor}\label{cor.L-decay-Lw}
Let $l$ be a non-negative integer. If $u_{0}\in L^{1}(\R^{2})$, then we have  
\begin{equation*}
\left\| \p_{x}^{l}K(t)*u_{0}\right\|_{L^{\infty}} \le C\left\|u_{0}\right\|_{L^{1}}t^{-\frac{5}{4}-\frac{l}{2}}, \ \ t>0. 
\end{equation*}
\end{cor}
Furthermore, we are able to show that $S(x, y, t)$ can be well approximated by $K(x, y, t)$ as follows. 
The following asymptotic formula is a direct generalization of Proposition~2.5 in \cite{FH23}. 
\begin{prop}\label{prop.L-ap-S}
Let $l$ and $m$ be non-negative integers. Then, we have 
\begin{equation}\label{L-ap-S}
\left\| \p_{x}^{l} \left( S(t) - \sum_{n=0}^{m}\frac{(-t)^{n}}{n!}\p_{x}^{3n}K(t) \right) \right\|_{L^{\infty}} \le C t^{ -\frac{7}{4} -\frac{l+m}{2} }, \ \ t>0, 
\end{equation}
where $S(x, y, t)$ and $K(x, y, t)$ are defined by \eqref{DEF-S} and \eqref{DEF-K}, respectively. 
\end{prop}
\begin{proof}
First, substituting $f(\theta)=e^{i\theta t\xi^{3}}$ into the following Taylor's theorem with integral remainder: 
\begin{align*}
f(1)=\sum_{n=0}^{m}\frac{1}{n!}f^{(n)}(0)+\frac{1}{m!}\int_{0}^{1}(1-\theta)^{m}f^{(m+1)}(\theta)d\theta, 
\end{align*}
then we have
\[
e^{it\xi^{3}}=\sum_{n=0}^{m}\frac{(it\xi^{3})^{n}}{n!}+\frac{(it\xi^{3})^{m+1}}{m!}\int_{0}^{1}(1-\theta)^{m}e^{i\theta t\xi^{3}}d\theta. 
\]
Therefore, it follows from the definitions of $S(x, y, t)$ in \eqref{DEF-S} and $K(x, y, t)$ in \eqref{DEF-K} that 
\begin{align*}
\hat{S}(\xi, \eta, t)
&=\frac{1}{2\pi}\sum_{n=0}^{m}\frac{(it\xi^{3})^{n}}{n!}e^{-\nu t\xi^{2}-it\e \frac{\eta^{2}}{\xi}}+\frac{(it\xi^{3})^{m+1}}{2\pi m!}e^{-\nu t\xi^{2}-it\e \frac{\eta^{2}}{\xi}}\int_{0}^{1}(1-\theta)^{m}e^{i\theta t\xi^{3}}d\theta \\
&=\sum_{n=0}^{m}\frac{(-t)^{n}}{n!}(i\xi)^{3n}\hat{K}(\xi, \eta, t)
+\frac{(it\xi^{3})^{m+1}}{2\pi m!}e^{-\nu t\xi^{2}-it\e \frac{\eta^{2}}{\xi}}\int_{0}^{1}(1-\theta)^{m}e^{i\theta t\xi^{3}}d\theta. 
\end{align*}

Now, taking the inverse Fourier transform of both sides of the above, we have 
\begin{align}
&S(x, y, t) - \sum_{n=0}^{m}\frac{(-t)^{n}}{n!}\p_{x}^{3n}K(x, y, t) \nonumber \\
&= \frac{(it)^{m+1}}{(2\pi)^{2}m!}\int_{0}^{1}(1-\theta)^{m}\left(\int_{\R}\int_{\R}\xi^{3(m+1)}e^{-\nu t\xi^{2}-it\e \frac{\eta^{2}}{\xi}+i\theta t\xi^{3}}e^{ix\xi+iy\eta}d\xi d\eta \right)d\theta \nonumber  \\
&= \frac{(it)^{m+1}}{(2\pi)^{\frac{3}{2}}m!}\int_{0}^{1}(1-\theta)^{m}\left(\int_{\R}\xi^{3(m+1)}e^{-\nu t\xi^{2}+i\theta t\xi^{3}+ix\xi}\mathcal{F}^{-1}_{\eta}\left[e^{-it\e \frac{\eta^{2}}{\xi}}\right](y)d\xi \right)d\theta  \nonumber \\
&=\frac{(it)^{m+1}t^{-\frac{1}{2}}}{4\pi^{\frac{3}{2}}m!}\int_{0}^{1}(1-\theta)^{m}\left(\int_{\R}|\xi|^{\frac{1}{2}}\xi^{3(m+1)}e^{-\nu t\xi^{2}+i\theta t\xi^{3}+ix\xi}e^{\frac{i\xi y^{2}}{4t\e}-\frac{i\pi}{4}\e \mathrm{sgn \xi}}d\xi \right)d\theta \nonumber  \\
&=:\Phi(x, y, t),  \label{DEF-R}
\end{align}
where we used the following fact: 
\begin{align}
\mathcal{F}^{-1}_{\eta} \left[ e^{-it\e \frac{\eta^{2}}{\xi}} \right](y) 
& = \frac{1}{\sqrt{2\pi}} \int_{\R}e^{-it\e \frac{\eta^{2}}{\xi}} e^{iy\eta} d\eta 
= \frac{1}{\sqrt{2\pi}} e^{ \frac{i\xi y^{2}}{4t\e} } \int_{\R} e^{ -\frac{it\e}{\xi}\left( \eta - \frac{\xi y}{2t\e} \right)^{2} } d\eta \nonumber \\
& = \frac{1}{\sqrt{2\pi}}\left( \frac{\xi \pi}{it\e} \right)^{\frac{1}{2}} e^{ \frac{i\xi y^{2}}{4t\e} }
= \sqrt{\frac{|\xi|}{2t}} e^{ \frac{i\xi y^{2}}{4t\e} - \frac{i\pi}{4}\e \mathrm{sgn} \xi }. \nonumber 
\end{align}

Finally, we directly evaluate the remainder term $\Phi(x, y, t)$ as follows: 
\begin{align}
\left|\p_{x}^{l}\Phi(x, y, t)\right|
&\le \frac{t^{\frac{1}{2}+m}}{2\pi^{\frac{3}{2}}m!}\int_{0}^{1}(1-\theta)^{m}d\theta \int_{0}^{\infty}\xi^{l+3m+\frac{7}{2}}e^{-\nu t\xi^{2}}d\xi \\
&=\frac{t^{-\frac{7}{4}-\frac{l+m}{2}}}{4\pi^{\frac{3}{2}}\nu^{\frac{l+3m}{2}+\frac{9}{4}}(m+1)!}\int_{0}^{\infty}r^{\frac{l+3m}{2}+\frac{5}{4}}e^{-r}dr \nonumber \\
&=\frac{\Gamma\left(\frac{l+3m}{2}+\frac{9}{4}\right)}{4\pi^{\frac{3}{2}}\nu^{\frac{l+3m}{2}+\frac{9}{4}}(m+1)!}t^{-\frac{7}{4}-\frac{l+m}{2}}, \ \ (x, y)\in \R^{2}, \ t>0. \label{est-R}
\end{align}
Thus, combining \eqref{DEF-R} and \eqref{est-R}, we arrive at the desired result \eqref{L-ap-S}. 
\end{proof}
\begin{cor}\label{cor.L-decay-LKPB-ap}
Let $l$ and $m$ be non-negative integers. If $u_{0}\in L^{1}(\R^{2})$, then we have  
\begin{equation*}
\left\| \p_{x}^{l}\left(S(t)*u_{0} - \sum_{n=0}^{m}\frac{(-t)^{n}}{n!}\p_{x}^{3n}K(t)*u_{0}\right) \right\|_{L^{\infty}} \le C\left\|u_{0}\right\|_{L^{1}}t^{-\frac{7}{4}-\frac{l+m}{2}}, \ \ t>0, 
\end{equation*}
where $S(x, y, t)$ and $K(x, y, t)$ are defined by \eqref{DEF-S} and \eqref{DEF-K}, respectively. 
\end{cor}

The above corollary means that the higher-order asymptotic profiles of $S(x, y, t)$ can be given by the derivatives of the solution $\psi(x, y, t)$ to \eqref{w-linear}. 
Next in this subsection, we shall give the more detailed asymptotic formula for the solution to \eqref{w-linear}. 
We analyzed the $x$-integration in the convolution in \eqref{w-sol} by employing a method similar to that used in the asymptotic expansion for solutions to parabolic equations, and derived the following formula \eqref{L-ap-LKPB}. This can be regarded as an anisotropic asymptotic expansion which generalizes Theorem~2.8 in \cite{FH23}. In order to give such a result, for $n \in \mathbb{N}\cup \{0\}$, let us introduce the following functions: 
\begin{align}
\begin{split}\label{DEF-mathK}
& \mathcal{K}_{n}(x, y, t) := \int_{\R} \p_{x}^{n}K(x, y-w, t)  M_{n}(w) dw, \ \ (x, y)\in \R^{2}, \ t>0,  \\
& M_{n}(y) := \frac{(-1)^{n}}{n!}\int_{\R} x^{n}u_{0}(x, y) dx, \ \ |\cdot|u_{0}(\cdot, y)\in L^{1}(\R), \ \mathrm{a.e.} \ y\in \R, 
\end{split}
\end{align}
where $K(x, y, t)$ is defined by \eqref{DEF-K}. Under this situation, an anisotropic asymptotic profile of the solution $\psi(x, y, t)$ to \eqref{w-linear} is given by the summation of $\mathcal{K}_{n}(x, y, t)$, if the initial data $u_{0}(x, y)$ satisfies some weight conditions. Actually, we can get the following result:  
\begin{prop}\label{prop.L-ap-LKPB}
Let $l$ and $m$ be non-negative integers. Suppose that the initial data $u_{0}(x, y)$ satisfies $x^{n}u_{0}\in L^{1}(\R^{2})$ for all $n=0, 1, \cdots, m$. 
Then, the following asymptotic formula holds: 
\begin{equation}\label{L-ap-LKPB}
\lim_{t \to \infty} t^{ \frac{5}{4}+\frac{l+m}{2} } \left\| \p_{x}^{l}\left(K(t)*u_{0}- \sum_{n=0}^{m}\mathcal{K}_{n}(t)\right) \right\|_{L^{\infty}} = 0. 
\end{equation}
Moreover, if we additionally assume $x^{m+1}u_{0}\in L^{1}(\R^{2})$, then we obtain  
\begin{equation}\label{L-ap-LKPB-2}
\left\| \p_{x}^{l}\left(K(t)*u_{0}- \sum_{n=0}^{m}\mathcal{K}_{n}(t)\right) \right\|_{L^{\infty}}\le C\left\|x^{m+1}u_{0}\right\|_{L^{1}}t^{ -\frac{5}{4}-\frac{l+m+1}{2} } , \ \ t>0, 
\end{equation}
where $K(x, y, t)$ and $\mathcal{K}_{n}(x, y, t)$ are defined by \eqref{DEF-K} and \eqref{DEF-mathK}, respectively. 
\end{prop}
\begin{proof}
We first consider the case of $m\ge1$. Recalling Taylor's theorem again, we can see that 
\begin{align*}
f(1)&=\sum_{n=0}^{m}\frac{1}{n!}f^{(n)}(0)+\frac{1}{m!}\int_{0}^{1}(1-\theta)^{m}f^{(m+1)}(\theta)d\theta \\
&=\sum_{n=0}^{m-1}\frac{1}{n!}f^{(n)}(0)+\frac{1}{(m-1)!}\int_{0}^{1}(1-\theta)^{m-1}f^{(m)}(\theta)d\theta.  
\end{align*}
Then, substituting $f(\theta)=K(x-\theta z, y-w, t)$ into the above equation, we obtain 
\begin{align}
&K(x-z, y-w, t)-\sum_{n=0}^{m}\frac{(-z)^{n}}{n!}\p_{x}^{n}K(x, y-w, t) \nonumber \\
&=\frac{(-z)^{m+1}}{m!}\int_{0}^{1}(1-\theta)^{m}\p_{x}^{m+1}K(x-\theta z, y-w, t)d\theta \nonumber \\
&=\frac{(-z)^{m}}{(m-1)!}\int_{0}^{1}(1-\theta)^{m-1}\p_{x}^{m}K(x-\theta z, y-w, t)d\theta-\frac{(-z)^{m}}{m!}\p_{x}^{m}K(x, y-w, t). \label{K-henkei}
\end{align}

Next, under the weight assumptions on the initial data $u_{0}(x, y)$, for any $\varepsilon_{0}>0$, there exists a constant $L=L(\varepsilon_{0})>0$ such that 
\begin{equation}\label{data-w}
\int_{|(z, w)|\ge L}|z^{m}u_{0}(z, w)|dzdw <\e_{0}.
\end{equation}

Now, based on the above result \eqref{K-henkei} and \eqref{DEF-mathK}, splitting the integral, we obtain the following equation for the target solution $\psi(x, y, t)$ and profile functions: 
\begin{align}
&\left(K(t)*u_{0}\right)(x, y)- \sum_{n=0}^{m}\mathcal{K}_{n}(x, y, t) \nonumber \\
&=\int_{\R^{2}}K(x-z, y-w, t)u_{0}(z, w) dzdw 
-\sum_{n=0}^{m}\frac{(-1)^{n}}{n!}\int_{\R}\p_{x}^{n}K(x, y-w, t)\left(\int_{\R}z^{n}u_{0}(z, w)dz\right)dw \nonumber \\
&=\int_{\R^{2}}\left\{K(x-z, y-w, t)
-\sum_{n=0}^{m}\frac{(-z)^{n}}{n!}\p_{x}^{n}K(x, y-w, t)\right\}u_{0}(z, w) dzdw \nonumber \\
&=\frac{(-1)^{m+1}}{m!}\int_{\R^{2}}\left(\int_{0}^{1}(1-\theta)^{m}\p_{x}^{m+1}K(x-\theta z, y-w, t)d\theta\right) z^{m+1}u_{0}(z, w)dzdw  \nonumber \\
&=\frac{(-1)^{m+1}}{m!}\int_{|(z, w)|\le L}\left(\int_{0}^{1}(1-\theta)^{m}\p_{x}^{m+1}K(x-\theta z, y-w, t)d\theta\right) z^{m+1}u_{0}(z, w)dzdw  \nonumber \\
&\ \ \ \ +\frac{(-1)^{m}}{(m-1)!}\int_{|(z, w)|\ge L}\left(\int_{0}^{1}(1-\theta)^{m-1}\p_{x}^{m}K(x-\theta z, y-w, t)d\theta\right) z^{m}u_{0}(z, w)dzdw \nonumber \\
&\ \ \ \ +\frac{(-1)^{m+1}}{m!}\int_{|(z, w)|\ge L}\p_{x}^{m}K(x, y-w, t)z^{m}u_{0}(z, w)dzdw.  \label{tuika}
\end{align}
Therefore, applying Proposition~\ref{prop.L-decay-K} and \eqref{data-w} to the above result, we can immediately have 
\begin{align*}
&\left\|\p_{x}^{l}\left(K(t)*u_{0} - \sum_{n=0}^{m}\mathcal{K}_{n}(t)\right)\right\|_{L^{\infty}} \\
&\le \int_{|(z, w)|\le L}\left(\int_{0}^{1}\left\|(1-\theta)^{m}\p_{x}^{l+m+1}K(\cdot -\theta z, y-w, t)\right\|_{L^{\infty}}d\theta\right) |z^{m+1}u_{0}(z, w)|dzdw  \\
&\ \ \ \ +\int_{|(z, w)|\ge L}\left(\int_{0}^{1}\left\|(1-\theta)^{m-1}\p_{x}^{l+m}K(\cdot-\theta z, y-w, t)\right\|_{L^{\infty}}d\theta\right) |z^{m}u_{0}(z, w)|dzdw \\
&\ \ \ \ +\int_{|(z, w)|\ge L}\left\|\p_{x}^{l+m}K(\cdot, y-w, t)\right\|_{L^{\infty}}|z^{m}u_{0}(z, w)|dzdw \\
&\le CL^{m+1}\left\|u_{0}\right\|_{L^{1}}t^{-\frac{5}{4}-\frac{l+m+1}{2}}+C\e_{0}t^{-\frac{5}{4}-\frac{l+m}{2}}, \ \ t>0. 
\end{align*}
As a result, we eventually arrive at 
\[
\limsup_{t\to \infty} t^{\frac{5}{4}+\frac{l+m}{2}}\left\|\p_{x}^{l}\left(K(t)*u_{0} - \sum_{n=0}^{m}\mathcal{K}_{n}(t)\right)\right\|_{L^{\infty}} \le C\e_{0}. 
\]
Thus, we get the first desired formula \eqref{L-ap-LKPB} for $m\ge1$, because $\varepsilon_{0}>0$ can be chosen arbitrarily small. 
On the other hand, if $m=0$, let us use the following relation instead of \eqref{tuika}: 
\begin{align*}
\left(K(t)*u_{0}\right)(x, y)-\mathcal{K}_{0}(x, y, t) 
&=-\int_{|(z, w)|\le L}\left(\int_{0}^{1}\p_{x}K(x-\theta z, y-w, t)d\theta \right)z u_{0}(z, w)dzdw \\
&\ \ \ \,+\int_{|(z, w)|\ge L}\left\{K(x-z, y-w, t)
-K(x, y-w, t)\right\}u_{0}(z, w) dzdw. 
\end{align*}
Then, we can see the similar result analogously as the above discussion. Actually, we have 
\[
\left\|\p_{x}^{l}\left(K(t)*u_{0} - \mathcal{K}_{0}(t)\right)\right\|_{L^{\infty}}
\le CL\left\|u_{0}\right\|_{L^{1}}t^{-\frac{5}{4}-\frac{l+1}{2}}+Ct^{-\frac{5}{4}-\frac{l}{2}}\int_{|(z, w)|\ge L}|u_{0}(z, w)|dzdw. 
\]
In addition, the proof of the second formula \eqref{L-ap-LKPB-2} can be given by a slight modification of the proof of the first one, and is therefore omitted. 
This completes the proof. 
\end{proof}

The following theorem can be obtained by combining Corollary~\ref{cor.L-decay-LKPB-ap} and Proposition~\ref{prop.L-ap-LKPB}, and is rewritten in the form required for the proof of the main theorem. Actually, the formula \eqref{L-ap-second1} plays an important role in the proof of the decay estimate of the solution, whereas the formula \eqref{L-ap-second2} is essential for its asymptotic expansion.
\begin{thm}\label{thm.L-ap-second}
Let $l$ be a non-negative integer. Suppose that the initial data $u_{0}(x, y)$ satisfies $(1+|x|)u_{0}\in L^{1}(\R^{2})$. 
Then, we have the following two asymptotic formulas: 
\begin{equation}\label{L-ap-second1}
\left\| \p_{x}^{l}\left(S(t)*u_{0}-\mathcal{K}_{0}(t)\right) \right\|_{L^{\infty}}\le C\left(\left\|u_{0}\right\|_{L^{1}}+\left\|xu_{0}\right\|_{L^{1}}\right)t^{ -\frac{7}{4}-\frac{l}{2} } , \ \ t>0, 
\end{equation}
\begin{equation}\label{L-ap-second2}
\lim_{t \to \infty} t^{ \frac{7}{4}+\frac{l}{2} } \left\| \p_{x}^{l}\left(S(t)*u_{0}- \mathcal{K}_{0}(t)-\mathcal{K}_{1}(t)+t\p_{x}^{3}\mathcal{K}_{0}(t)\right) \right\|_{L^{\infty}} = 0, 
\end{equation}
where $S(x, y, t)$ and $\mathcal{K}_{n}(x, y, t)$ for $n=0, 1$ are defined by \eqref{DEF-S} and \eqref{DEF-mathK}, respectively. 
\end{thm}
\begin{proof}
Applying Corollary~\ref{cor.L-decay-LKPB-ap} and \eqref{L-ap-LKPB-2} for $m=0$, we immediately get \eqref{L-ap-second1} as follows: 
\begin{align*}
\left\| \p_{x}^{l}\left(S(t)*u_{0}-\mathcal{K}_{0}(t)\right) \right\|_{L^{\infty}}
&\le \left\|\p_{x}^{l}\left(S(t)*u_{0}-K(t)*u_{0}\right) \right\|_{L^{\infty}}+\left\| \p_{x}^{l}\left(K(t)*u_{0}-\mathcal{K}_{0}(t)\right) \right\|_{L^{\infty}} \\
&\le C\left\|u_{0}\right\|_{L^{1}}t^{ -\frac{7}{4}-\frac{l}{2} }+C\left\|xu_{0}\right\|_{L^{1}}t^{ -\frac{7}{4}-\frac{l}{2} }, \ \ t>0. 
\end{align*}

In order to prove the second asymptotic formula, we note that the following equation holds: 
\begin{align*}
&\left(S(t)*u_{0}\right)(x, y)-\mathcal{K}_{0}(x, y, t)-\mathcal{K}_{1}(x, y, t)+t\p_{x}^{3}\mathcal{K}_{0}(x, y, t) \\
&=\left(S(t)*u_{0}\right)(x, y)-\left(K(t)*u_{0}\right)(x, y)+t\p_{x}^{3}\left(K(t)*u_{0}\right)(x, y) \\
&\ \ \ \ -t\p_{x}^{3}\left\{\left(K(t)*u_{0}\right)(x, y)-\mathcal{K}_{0}(x, y, t)\right\}
+\left(K(t)*u_{0}\right)(x, y)-\mathcal{K}_{0}(x, y, t)-\mathcal{K}_{1}(x, y, t). 
\end{align*}
In the right hand side of the above, applying Corollary~\ref{cor.L-decay-LKPB-ap} for $m=1$ to the first term, we have 
\[
\left\|\p_{x}^{l}\left(S(t)*u_{0}-K(t)*u_{0}+t\p_{x}^{3}K(t)*u_{0}\right)\right\|_{L^{\infty}} \le C\left\|u_{0}\right\|_{L^{1}}t^{-\frac{9}{4}-\frac{l}{2}}, \ \ t>0. 
\]
On the other hand, using \eqref{L-ap-LKPB-2} in the case of $m=0$, we obtain 
\[
\left\|t\p_{x}^{3+l}\left(K(t)*u_{0}-\mathcal{K}_{0}(t)\right)\right\|_{L^{\infty}} \le C\left\|xu_{0}\right\|_{L^{1}}t^{-\frac{9}{4}-\frac{l}{2}}, \ \ t>0. 
\]
Finally, it follows from \eqref{L-ap-LKPB} with $m=1$ that 
\[
\lim_{t\to \infty}t^{\frac{7}{4}+\frac{l}{2}}\left\|\p_{x}^{l}\left(K(t)*u_{0}-\mathcal{K}_{0}(t)-\mathcal{K}_{1}(t)\right)\right\|_{L^{\infty}}=0. 
\]
Thus, \eqref{L-ap-second2} follows from all the above three expressions and taking the limit as $t\to \infty$.
\end{proof}

\subsection{Nonlinear Analysis}  

\indent

In this subsection, we would like to prepare some estimates for the solution to the original problem \eqref{KPB} and give an auxiliary lemma related to the nonlinear analysis. First, let us introduce a priori estimate of the solution $u(x, y, t)$ given by Molinet~\cite{M99}. The following proposition can be easily shown by the standard energy method (for the proof, see Theorem~3.1 in \cite{M99}).
\begin{prop}\label{prop.apriori}
Let $p\ge1$ be an integer. 
Assume that $u_{0}\in X^{3}(\R^{2})$ and $B(u_{0})$ is sufficiently small. 
Then, the solution $u(x, y, t)$ to \eqref{KPB} satisfies the following estimate: 
\begin{align}\label{apriori}
B(u(t))+\int_{0}^{t}\left(\left\|u_{y}(\tau)\right\|_{L^{2}}^{2}+\left\|u_{x}(\tau)\right\|_{H^{1}}^{2}+\left\|u_{xxx}(\tau)\right\|_{L^{2}}^{2}\right)d\tau \le CB(u_{0}), \ \ t\ge0, 
\end{align}
where $B(u(t))$ is defined by $B(u(t)):=\left\|u(t)\right\|_{H^{1}}^{2}+\left\|\p_{x}^{2}u(t)\right\|_{L^{2}}^{2}+\left\|\p_{x}^{-1}\p_{y}u(t)\right\|_{L^{2}}^{2}$. 
\end{prop}

\begin{rem}
{\rm 
By virtue of \eqref{apriori} and the inequality 
\begin{equation*}
\left\|f\right\|_{L^{\infty}}\le C\left( \left\|f\right\|_{L^{2}}+\left\|f_{y}\right\|_{L^{2}}+\left\|f_{xx}\right\|_{L^{2}} \right)
\end{equation*}
(cf.~\cite{BIN78}), we have the uniform boundedness for the $L^{\infty}$-norm of the solution to \eqref{KPB} as follows: 
\begin{equation}\label{Linf-bdd}
\left\|u(t)\right\|_{L^{\infty}}\le C\sqrt{B(u_{0})}, \ \ t\ge0. 
\end{equation}
This fact will be often used in the latter part of the paper. 
}
\end{rem}

The next results introduced below are an additional $L^{2}$-estimate and an asymptotic decay of the solution to \eqref{KPB}. 
These results have already been established in Molinet~\cite{M99} (see its Corollary~3.2 and Lemmas~5.1 and 5.2). 
For the convenience of the reader, we extract the part relevant to the present study and provide their proofs here. In what follows, we assume that $B(u_{0})\le 1$. 
\begin{prop}\label{prop.L2-decay}
Let $p\ge1$ be an integer. 
Assume that $u_{0}\in X^{3}(\R^{2})$ and $B(u_{0})$ is sufficiently small. 
Then, the solution $u(x, y, t)$ to \eqref{KPB} satisfies the following estimate: 
\begin{equation}\label{u-L2L2-est}
\left\|\p_{x}^{-1}u(t)\right\|_{L^{2}}^{2}+\nu  \int_{0}^{t}\left\|u(\tau)\right\|_{L^{2}}^{2}d\tau
\le C\left\|\p_{x}^{-1}u_{0}\right\|_{L^{2}}^{2}, \ \ t\ge0. 
\end{equation}
Moreover, the following asymptotic formula also holds: 
\begin{equation}\label{u-sol-decay-L2}
\lim_{t\to \infty}t^{\frac{1}{2}}\left(\left\|u(t)\right\|_{L^{2}}^{2}+\left\|u_{y}(t)\right\|_{L^{2}}^{2}\right)^{\frac{1}{2}}=0.  
\end{equation}
\end{prop}
\begin{proof}
First, we shall prove \eqref{u-L2L2-est}. Applying the operator $\p_{x}^{-1}$ to the original equation \eqref{KPB} and multiplying $\p_{x}^{-1}u$ on \eqref{KPB}, and then integrating the resulting equation over $\R^{2}$, we have 
\begin{align*}
\frac{d}{dt}\left\|\p_{x}^{-1}u(t)\right\|_{L^{2}}+2\nu  \left\|u(t)\right\|_{L^{2}} 
&\le \frac{2}{p+1}\int_{\R^{2}}\left|u^{p+1}\p_{x}^{-1}u\right|dxdy
\le \left\|u(t)\right\|_{L^{\infty}}^{p-1}\left\|u(t)\right\|_{L^{4}}^{2}\left\|\p_{x}^{-1}u(t)\right\|_{L^{2}} \\
&\le CB(u_{0})^{\frac{p-1}{2}}\left\|u(t)\right\|_{L^{2}}\left\|u_{x}(t)\right\|_{L^{2}}^{\frac{1}{2}}\left\|u_{y}(t)\right\|_{L^{2}}^{\frac{1}{2}}\left\|\p_{x}^{-1}u(t)\right\|_{L^{2}}
\\
&\le \nu \left\|u(t)\right\|_{L^{2}}^{2}+C\left(\left\|u_{x}(t)\right\|_{L^{2}}^{2}+\left\|u_{y}(t)\right\|_{L^{2}}^{2}\right)\left\|\p_{x}^{-1}u(t)\right\|_{L^{2}}^{2}, \ \ t>0, 
\end{align*}
where we used the Cauchy--Schwarz inequality, \eqref{Linf-bdd} and the following inequality: 
\[
\left\|f\right\|_{L^{2q}}^{2q}
\le (q!)^{2}\left\|f\right\|_{L^{2}}^{2} \left\|f_{x}\right\|_{L^{2}}^{q-1} \left\|f_{y}\right\|_{L^{2}}^{q-1}, \ \ q\in \mathbb{N}
\]
(cf.~\cite{FH24}). Therefore, integrating the above inequality on $[0, t]$, we obtain 
\begin{align*}
&\left\|\p_{x}^{-1}u(t)\right\|_{L^{2}}^{2}+\nu  \int_{0}^{t}\left\|u(\tau)\right\|_{L^{2}}^{2}d\tau  \\
&\le \left\|\p_{x}^{-1}u_{0}\right\|_{L^{2}}^{2}
+C\int_{0}^{t}\left(\left\|u_{x}(\tau)\right\|_{L^{2}}^{2}+\left\|u_{y}(\tau)\right\|_{L^{2}}^{2}\right)\left\|\p_{x}^{-1}u(\tau)\right\|_{L^{2}}^{2}d\tau, \ \ t\ge0. 
\end{align*}
Thus, we can prove \eqref{u-L2L2-est} from Gronwall's lemma and a priori estimate \eqref{apriori} as follows: 
\begin{align*}
\left\|\p_{x}^{-1}u(t)\right\|_{L^{2}}^{2}+\nu  \int_{0}^{t}\left\|u(\tau)\right\|_{L^{2}}^{2}d\tau
&\le \left\|\p_{x}^{-1}u_{0}\right\|_{L^{2}}^{2}\exp\left\{C\int_{0}^{t}\left(\left\|u_{x}(\tau)\right\|_{L^{2}}^{2}+\left\|u_{y}(\tau)\right\|_{L^{2}}^{2}\right)d\tau\right\} \\
&\le \left\|\p_{x}^{-1}u_{0}\right\|_{L^{2}}^{2}\exp\left\{CB(u_{0})\right\} 
\le C\left\|\p_{x}^{-1}u_{0}\right\|_{L^{2}}^{2}, \ \ t\ge0. 
\end{align*}

Next, we would like to prove \eqref{u-sol-decay-L2}. Multiplying the solution $u$ on the both sides of \eqref{KPB} and integrating over $\R^{2}$, we immediately obtain 
\begin{equation}\label{multi-u}
\frac{d}{dt}\left\|u(t)\right\|_{L^{2}}^{2}+2\nu \left\|u_{x}(t)\right\|_{L^{2}}^{2}=0, \ \ t>0. 
\end{equation}
Moreover, multiplying $-u_{yy}$ on the both sides of \eqref{KPB} and integrating over $\R^{2}$, then it follows from the integration by parts twice that 
\begin{equation}\label{multi-uyy}
\frac{d}{dt}\left\|u_{y}(t)\right\|_{L^{2}}^{2}+2\nu \left\|u_{xy}(t)\right\|_{L^{2}}^{2}
=2\int_{\R^{2}}u^{p}u_{x}u_{yy}dxdy=2\int_{\R^{2}}u^{p}u_{y}u_{xy}dxdy, \ \ t>0. 
\end{equation}
Thus, combining \eqref{multi-u} and \eqref{multi-uyy}, and using the Cauchy--Schwarz inequality, the inequality 
\[
\left\|f\right\|_{L^{2(q+1)}}^{2(q+1)}\le C \left\|f\right\|_{L^{2}}^{2-q}\left\|f_{x}\right\|_{L^{2}}^{2q}\left\|\p_{x}^{-1}f_{y}\right\|_{L^{2}}^{q}, \ \ q\in [0, 2)
\]
(cf.~\cite{BIN78}), the $L^{\infty}$-boundedness \eqref{Linf-bdd} and a priori estimate \eqref{apriori-Xs}, we can get the following result: 
\begin{align}
&\frac{d}{dt}\left(\left\|u(t)\right\|_{L^{2}}^{2}+\left\|u_{y}(t)\right\|_{L^{2}}^{2}\right)+2\nu \left(\left\|u_{x}(t)\right\|_{L^{2}}^{2}+\left\|u_{xy}(t)\right\|_{L^{2}}^{2}\right) \nonumber \\
&\le 2\int_{\R^{2}}\left|u^{p}u_{y}u_{xy}\right|dxdy 
\le 2\left\|u(t)\right\|_{L^{\infty}}^{p-1}\left\|u(t)\right\|_{L^{4}} \left\|u_{y}(t)\right\|_{L^{4}} \left\|u_{xy}(t)\right\|_{L^{2}} \nonumber \\
&\le CB(u_{0})^{\frac{p-1}{2}}\left\|u(t)\right\|_{L^{2}}^{\frac{1}{4}}\left\|u_{x}(t)\right\|_{L^{2}}^{\frac{1}{2}}\left\|\p_{x}^{-1}u_{y}(t)\right\|_{L^{2}}^{\frac{1}{4}}
\left\|u_{y}(t)\right\|_{L^{2}}^{\frac{1}{4}}\left\|u_{xy}(t)\right\|_{L^{2}}^{\frac{1}{2}}\left\|\p_{x}^{-1}u_{yy}(t)\right\|_{L^{2}}^{\frac{1}{4}}\left\|u_{xy}(t)\right\|_{L^{2}} \nonumber \\
&\le C\left\|u(t)\right\|_{L^{2}}^{\frac{1}{4}}\left(\left\|u_{y}(t)\right\|_{L^{2}}^{\frac{1}{2}}\left\|\p_{x}^{-1}u_{y}(t)\right\|_{L^{2}}^{\frac{1}{2}}
\left\|\p_{x}^{-1}u_{yy}(t)\right\|_{L^{2}}^{\frac{1}{2}}\left\|u_{x}(t)\right\|_{L^{2}}\left\|u_{xy}(t)\right\|_{L^{2}}+\left\|u_{xy}(t)\right\|_{L^{2}}^{2}\right) \nonumber \\
&\le C\left\|u(t)\right\|_{L^{2}}^{\frac{1}{4}}\left\{C_{\dag}\left(B(u_{0}), \|u_{0}\|_{X^{s}}\right)\left(\left\|u_{x}(t)\right\|_{L^{2}}^{2}+\left\|u_{xy}(t)\right\|_{L^{2}}^{2}\right)+\left\|u_{xy}(t)\right\|_{L^{2}}^{2}\right\} \nonumber \\
&\le C\left\|u(t)\right\|_{L^{2}}^{\frac{1}{4}}\left(\left\|u_{x}(t)\right\|_{L^{2}}^{2}+\left\|u_{xy}(t)\right\|_{L^{2}}^{2}\right), \ \ t>0. \label{u+uy-est}
\end{align}

Now, we set $g(t):=\left\|u(t)\right\|_{L^{2}}^{2}+\left\|u_{y}(t)\right\|_{L^{2}}^{2}$. Then, we note that $g, g' \in L^{1}(0, \infty)$ and thus the fact $g(t) \to 0$ ($t\to \infty$) is true. Indeed, modifying the above argument and using \eqref{apriori}, we have 
\[
\left\|g'\right\|_{L^{1}(0, \infty)}=\int_{0}^{\infty}\left|\frac{d}{dt}\left(\left\|u(t)\right\|_{L^{2}}^{2}+\left\|u_{y}(t)\right\|_{L^{2}}^{2}\right)\right|dt
\le C\int_{0}^{\infty}\left(\left\|u_{x}(t)\right\|_{L^{2}}^{2}+\left\|u_{xy}(t)\right\|_{L^{2}}^{2}\right)dt<\infty. 
\]
On the other hand, $g \in L^{1}(0, \infty)$ is a direct consequence of the first result \eqref{u-L2L2-est} and a priori estimate \eqref{apriori}. 
Thus, we have $g(t) \to 0$ ($t\to \infty$), and therefore $\|u(t)\|_{L^{2}}^{1/4}\to 0$ ($t\to \infty$) holds. 
Then, by virtue of \eqref{u+uy-est}, there exists $T>0$ such that 
\[
\frac{d}{dt}\left(\left\|u(t)\right\|_{L^{2}}^{2}+\left\|u_{y}(t)\right\|_{L^{2}}^{2}\right) + \nu \left(\left\|u_{x}(t)\right\|_{L^{2}}^{2}+\left\|u_{xy}(t)\right\|_{L^{2}}^{2}\right)\le 0, \ \ t\ge T. 
\]
Hence, we can say that $g(t)$ is a monotone non-increasing function on $[T, \infty)$. Therefore, for any $t\in [T, \infty)$ and $\tau \in (t, \infty)$, we obtain the following fact: 
\[
\int_{t}^{\tau}g(s)ds\ge (\tau-t)g(\tau). 
\]
Thus, for fixed $t\in [T, \infty)$, it follows from the fact $g(\tau) \to 0$ ($\tau \to \infty$) that 
\[
\int_{t}^{\infty}g(s)ds\ge \limsup_{t\to \infty}\tau g(\tau). 
\]
Finally, since $g \in L^{1}(0, \infty)$, the left hand side of the above converges to zero as $t \to \infty$. 
As a result, we have the fact $\tau g(\tau) \to 0$ ($\tau \to \infty$). Namely, we are able to see that 
\begin{equation*}
\lim_{t \to \infty} t g(t)=\lim_{t \to \infty} t \left(\left\|u(t)\right\|_{L^{2}}^{2}+\left\|\p_{y}u(t)\right\|_{L^{2}}^{2} \right)=0. 
\end{equation*}
Taking the roots of both sides of this result, we can get the second desired formula \eqref{u-sol-decay-L2}. 
\end{proof}

Finally in this subsection, let us prepare an auxiliary estimate which is a revised version of Lemma~3.5 in \cite{FH23}. 
The result below is related to the analysis for the Duhamel term in the integral equation \eqref{integral-eq} and plays a crucial role throughout the proof of Theorem~\ref{thm.main}. Although this is a simple estimate, it should be noted that the derivative loss due to the nonlinear term that appeared in the result given in \cite{FH23} is slightly recovered by a small modification of the original proof.
\begin{lem}\label{lem.Duhamel-est-Linf}
Let $t>a>0$ and $g\in C((0, \infty); (H^{0, 1} \cap L^{1})(\R^{2}))$. Then, we have 
\begin{align}\label{Duhamel-est-Linf}
\begin{split}
&\left\|\int_{a}^{t}\p_{x}S(t-\tau)*g(\tau)d\tau\right\|_{L^{\infty}}+\left\|\int_{a}^{t}\p_{x}K(t-\tau)*g(\tau)d\tau\right\|_{L^{\infty}} \\
&\le \frac{1}{2^{\frac{5}{4}}\pi^{\frac{1}{4}}\nu^{\frac{3}{4}}}\int_{a}^{t}(t-\tau)^{-\frac{3}{4}}\left(\left\|g(\tau)\right\|_{L^{2}}^{2}+\left\|\p_{y}g(\tau)\right\|_{L^{2}}^{2}\right)^{\frac{1}{2}} d\tau,
\end{split}
\end{align}
where $S(x, y, t)$ and $K(x, y, t)$ are defined by \eqref{DEF-S} and \eqref{DEF-K}, respectively. 
\end{lem}
\begin{proof}
It follows from the Fourier transform and the Cauchy--Schwarz inequality that 
\begin{align*}
&\left\|\int_{a}^{t}\p_{x}S(t-\tau)*g(\tau)d\tau\right\|_{L^{\infty}} \nonumber \\
&=\frac{1}{2\pi} \sup_{(x, y)\in \R^{2}}\left| \int_{a}^{t}\int_{\R^{2}} (i\xi)e^{-\nu (t-\tau)\xi^{2}+i(t-\tau)\left(\xi^{3} - \e \frac{\eta^{2}}{\xi} \right)+ix\xi+iy\eta}\hat{g}(\xi, \eta, \tau)d\xi d\eta d\tau\right| \nonumber \\
&\le \frac{1}{2\pi} \int_{a}^{t}\left(\int_{\R}\xi^{2}e^{-2\nu (t-\tau)\xi^{2}}d\xi\right)^{\frac{1}{2}}\left(\int_{\R}\frac{d\eta}{1+\eta^{2}}\right)^{\frac{1}{2}}
\left(\int_{\R^{2}}\left(1+\eta^{2}\right) \left|\hat{g}(\xi, \eta, \tau)\right|^{2}d\xi d\eta\right)^{\frac{1}{2}}d\tau \nonumber \\
&=\frac{1}{2^{\frac{9}{4}}\pi^{\frac{1}{4}}\nu^{\frac{3}{4}}} \int_{a}^{t}(t-\tau)^{-\frac{3}{4}}\left(\left\|g(\tau)\right\|_{L^{2}}^{2}+\left\|\p_{y}g(\tau)\right\|_{L^{2}}^{2}\right)^{\frac{1}{2}}d\tau.
\end{align*}
If we replace $S(x, y, t)$ with $K(x, y, t)$ in the above proof, we can also get exactly the same estimate. 
Combining these two estimates, we can complete the proof of the desired result \eqref{Duhamel-est-Linf}. 
\end{proof}

\section{Proof of the Main Result}  

\indent

In this section, we shall prove Theorem~\ref{thm.main}. 
The proof of this theorem is carried out by dividing it into several steps. For this purpose, this section is divided into the part on the decay estimate of the solution, i.e., \eqref{u-sol-decay} and the part on the asymptotic behavior of the solution, i.e., \eqref{u-sol-asymp}.

\subsection{Decay Estimate}  

\indent

In this subsection, let us derive the $L^{\infty}$-decay estimate of the solution $u(x, y, t)$ to \eqref{KPB}. Namely, we shall prove \eqref{u-sol-decay}. 
For proving that result, we need to introduce the following quantity: 
\begin{equation}\label{DEF-H}
H(t):=\sup_{0\le \tau \le t}(1+\tau)^{\frac{7}{4}}\left\|u(\tau)\right\|_{L^{\infty}}, \ \ t\ge0. 
\end{equation}
Then, in order to prove \eqref{u-sol-decay}, it suffices to show the following proposition: 
\begin{prop}\label{prop.main-decay}
Let $p\ge1$ be an integer. Assume that $u_{0}\in X^{3}(\R^{2})\cap L^{1}(\R^{2})$, $xu_{0}\in L^{1}(\R^{2})$ and $B(u_{0})$ is sufficiently small. 
Then, we have the following estimate: 
\begin{equation}\label{main-decay}
H(t)\le CE_{0}, \ \ t\ge0, 
\end{equation}
where $H(t)$ and $E_{0}$ are defined by \eqref{DEF-H} and \eqref{data-1}, respectively. 
\end{prop}
\begin{proof}
First, splitting the $\tau$-integral of the integral equation \eqref{integral-eq}, we have 
\begin{align}
u(x, y, t)
&= S(t)*u_{0}-\frac{1}{p+1}\left(\int_{0}^{\frac{t}{2}}+\int_{\frac{t}{2}}^{t}\right)\p_{x}S(t-\tau)*u^{p+1}(\tau)d\tau  \nonumber \\
&=:I_{1}(x, y, t)+I_{2}(x, y, t)+I_{3}(x, y, t). \label{u-split}
\end{align}
In what follows, we shall evaluate $I_{1}(x, y, t)$, $I_{2}(x, y, t)$ and $I_{3}(x, y, t)$. We start with evaluating $I_{1}(x, y, t)$. 
Now, we note that $M_{0}(y)\equiv 0$ and thus $\mathcal{K}_{0}(x, y, t)\equiv 0$ holds (recall the definition \eqref{DEF-mathK} of these functions), under the zero-mass condition \eqref{zero-mass}. Therefore, it follows from \eqref{L-ap-second1} that 
\begin{equation}\label{I1-est}
 \left\|I_{1}(t)\right\|_{L^{\infty}}=\left\|S(t)*u_{0}\right\|_{L^{\infty}}\le C\left(\left\|u_{0}\right\|_{L^{1}}+\left\|xu_{0}\right\|_{L^{1}}\right)t^{-\frac{7}{4}}, \ \ t>0. 
\end{equation}

Next, we would like to deal with $I_{2}(x, y, t)$. Using Corollary~\ref{cor.L-decay-LKPB}, \eqref{Linf-bdd} and \eqref{u-L2L2-est}, we are able to evaluate $I_{2}(x, y, t)$ as follows: 
\begin{align}
\left\|I_{2}(t)\right\|_{L^{\infty}}&\le C\int_{0}^{\frac{t}{2}}(t-\tau)^{-\frac{7}{4}}\left\|u^{p+1}(\tau)\right\|_{L^{1}}d\tau  \nonumber \\
&\le C\left(\sup_{\tau\ge0}\left\|u(\tau)\right\|_{L^{\infty}}\right)^{p-1}\int_{0}^{\frac{t}{2}}(t-\tau)^{-\frac{7}{4}}\left\|u(\tau)\right\|_{L^{2}}^{2}d\tau \nonumber \\
&\le CB(u_{0})^{\frac{p-1}{2}}\left\|\p_{x}^{-1}u_{0}\right\|_{L^{2}}^{2}t^{-\frac{7}{4}}
\le C\left\|\p_{x}^{-1}u_{0}\right\|_{L^{2}}^{2}t^{-\frac{7}{4}}, \ \ t>0. \label{I2-est}
\end{align}

Finally, let us treat $I_{3}(x, y, t)$. Before evaluating it, we shall prepare an estimate for $L^{2}$-norm of the solution. 
Here, we write the positive constant $C$ appeared in \eqref{Linf-bdd} as $C_{\dag}$ for the latter sake. Moreover, in what follows, we assume that $B(u_{0})$ be sufficiently small such that $C_{\dag}\sqrt{B(u_{0})}\le1$. Now, from \eqref{u-sol-decay-L2}, there exists a sufficiently large $T_{\nu} \ge1$ which depending only $\nu>0$ such that 
\begin{equation}\label{L2-re}
\left(\left\|u(\tau)\right\|_{L^{2}}^{2}+\left\|u_{y}(\tau)\right\|_{L^{2}}^{2}\right)^{\frac{1}{2}} \le \frac{\pi^{\frac{1}{4}}\nu^{\frac{3}{4}}}{2^{\frac{15}{4}}}\tau^{-\frac{1}{2}}, \ \ \tau \ge T_{\nu}. 
\end{equation}
Then, applying Lemma~\ref{lem.Duhamel-est-Linf} to $I_{3}(x, y, t)$, we have from \eqref{DEF-H} and \eqref{L2-re} that 
\begin{align}
\left\|I_{3}(t)\right\|_{L^{\infty}}&\le \frac{1}{2^{\frac{5}{4}}\pi^{\frac{1}{4}}\nu^{\frac{3}{4}}(p+1)}\int_{\frac{t}{2}}^{t}(t-\tau)^{-\frac{3}{4}}\left(\left\|u^{p+1}(\tau)\right\|_{L^{2}}^{2}+\left\|\p_{y}\left(u^{p+1}(\tau)\right)\right\|_{L^{2}}^{2}\right)^{\frac{1}{2}}d\tau \nonumber \\
&\le \frac{1}{2^{\frac{5}{4}}\pi^{\frac{1}{4}}\nu^{\frac{3}{4}}}\int_{\frac{t}{2}}^{t}(t-\tau)^{-\frac{3}{4}}\left\|u(\tau)\right\|_{L^{\infty}}^{p-1}\left\|u(\tau)\right\|_{L^{\infty}}\left(\left\|u(\tau)\right\|_{L^{2}}^{2}+\left\|u_{y}(\tau)\right\|_{L^{2}}^{2}\right)^{\frac{1}{2}}d\tau \nonumber \\
&\le \frac{1}{2^{\frac{5}{4}}\pi^{\frac{1}{4}}\nu^{\frac{3}{4}}}\left(\sup_{\tau\ge0}\left\|u(\tau)\right\|_{L^{\infty}}\right)^{p-1}\int_{\frac{t}{2}}^{t}(t-\tau)^{-\frac{3}{4}}\cdot H(t)(1+\tau)^{-\frac{7}{4}}\cdot \frac{\pi^{\frac{1}{4}}\nu^{\frac{3}{4}}}{2^{\frac{15}{4}}}\tau^{-\frac{1}{2}}d\tau \nonumber \\
&\le \frac{1}{2^{5}}\left(C_{\dag}\sqrt{B(u_{0})}\right)^{\frac{p-1}{2}}H(t)\int_{\frac{t}{2}}^{t}(t-\tau)^{-\frac{3}{4}}\tau^{-\frac{9}{4}}d\tau 
\le \frac{1}{2^{5}}\left(\frac{t}{2}\right)^{-\frac{9}{4}}H(t)\int_{\frac{t}{2}}^{t}(t-\tau)^{-\frac{3}{4}}d\tau  \nonumber \\
&= \frac{1}{2^{5}}\left(\frac{t}{2}\right)^{-\frac{9}{4}}\cdot 4\left(\frac{t}{2}\right)^{\frac{1}{4}}\cdot H(t)
=\frac{1}{2}H(t)t^{-2}, \ \ t\ge 2T_{\nu}. \label{I3-est} 
\end{align}

Finally, for $0\le t \le 2T_{\nu}$, we note that the boundedness for the $L^{\infty}$-norm of the solution $u(x, y, t)$ can be guaranteed by \eqref{Linf-bdd}. 
Actually, we have 
\[
(1+t)^{\frac{7}{4}}\left\|u(t)\right\|_{L^{\infty}} \le C_{\dag}(1+2T_{\nu})^{\frac{7}{4}}\sqrt{B(u_{0})}. 
\]
Therefore, combining this fact with \eqref{u-split}, \eqref{I1-est}, \eqref{I2-est} and \eqref{I3-est}, there exists a constant $C_{0}>0$ independent of $t>0$ such that 
\[
H(t)\le C_{0}\left(\sqrt{B(u_{0})}+\left\|u_{0}\right\|_{L^{1}}+\left\|xu_{0}\right\|_{L^{1}}+\left\|\p_{x}^{-1}u_{0}\right\|_{L^{2}}^{2}\right)+\frac{1}{2}H(t), \ \ t\ge0. 
\]
Thus, we can eventually arrive at the following result: 
\[
H(t)\le 2C_{0}\left(\sqrt{B(u_{0})}+\left\|u_{0}\right\|_{L^{1}}+\left\|xu_{0}\right\|_{L^{1}}+\left\|\p_{x}^{-1}u_{0}\right\|_{L^{2}}^{2}\right)=2C_{0}E_{0}, \ \ t\ge0. 
\]
This means that the proof of the desired result \eqref{main-decay} has been completed. 
\end{proof}

\begin{rem}
{\rm 
The above proof also works under the assumption $\p_{x}^{-1}u_{0}\in L^{1}(\R^{2})$ like those of \cite{M99} and \cite{FH23}, instead of $xu_{0}\in L^{1}(\R^{2})$. More precisely, if we assume that $p\ge1$ be an integer, $u_{0}\in X^{3}(\R^{2})$, $\p_{x}^{-1}u_{0}\in L^{1}(\R^{2})$ and $B(u_{0})$ is sufficiently small, then we have the following uniform estimate: 
\begin{equation*}
H(t)\le C\left(\sqrt{B(u_{0})}+\left\|\p_{x}^{-1}u_{0}\right\|_{L^{1}}+\left\|\p_{x}^{-1}u_{0}\right\|_{L^{2}}^{2}\right), \ \ t\ge0.
\end{equation*}
Actually, under the condition $\p_{x}^{-1}u_{0}\in L^{1}(\R^{2})$, we can see that 
\begin{equation*}
\left\|S(t)*u_{0}\right\|_{L^{\infty}}\le C\left\|\p_{x}^{-1}u_{0}\right\|_{L^{1}}t^{-\frac{7}{4}}, \ \ t>0 
\end{equation*}
holds instead of \eqref{I1-est}. Combining this estimate and \eqref{u-split}, \eqref{I2-est} and \eqref{I3-est}, we can derive the desired result. 
As a conclusion, we arrive at 
\[
\left\|u(t)\right\|_{L^{\infty}}\le C\left(\sqrt{B(u_{0})}+\left\|\p_{x}^{-1}u_{0}\right\|_{L^{1}}+\left\|\p_{x}^{-1}u_{0}\right\|_{L^{2}}^{2}\right)(1+t)^{-\frac{7}{4}}, \ \ t\ge0. 
\] 
This means that we can generalize the known result given by Molinet (Theorem~5.2 in \cite{M99}). 
}
\end{rem}

\subsection{Asymptotic Behavior}  

\indent

Finally, we shall derive the asymptotic profile of the solution $u(x, y, t)$ to \eqref{KPB}, i.e., let us prove the formula \eqref{u-sol-asymp}. 
To do that, let us rewrite the integral equation \eqref{integral-eq} as follows: 
\begin{align}
u(t)&=S(t)*u_{0}-\frac{1}{p+1}\int_{0}^{t}\p_{x}S(t-\tau)*u^{p+1}(\tau)d\tau \nonumber \\
&=S(t)*u_{0}-\mathcal{K}_{1}(t)-\frac{1}{p+1}\int_{0}^{t}\p_{x}(S-K)(t-\tau)*u^{p+1}(\tau)d\tau \nonumber\\
&\ \ \ +\mathcal{K}_{1}(t)-\frac{1}{p+1}\int_{0}^{t}\p_{x}K(t-\tau)*u^{p+1}(\tau)d\tau. \label{IE-re}
\end{align}

\noindent
\underline{\bf Step 1. Linear Approximation}

\smallskip
In what follows, we would like to analyze all terms in the right hand side of \eqref{IE-re}. 
For the difference of the first and the second terms, i.e., $S(t)*u_{0}-\mathcal{K}_{1}(t)$, using \eqref{L-ap-second2} under the zero-mass condition \eqref{zero-mass}, in the same way to get \eqref{I1-est}, we can immediately have the following formula: 
\begin{cor}\label{cor.asymp-linear}
Suppose that the initial data $u_{0}(x, y)$ satisfies $(1+|x|)u_{0}\in L^{1}(\R^{2})$ and the zero-mass condition \eqref{zero-mass}. 
Then, we have the following asymptotic formula: 
\begin{equation*}
\lim_{t \to \infty} t^{ \frac{7}{4}} \left\| S(t)*u_{0}-\mathcal{K}_{1}(t)\right\|_{L^{\infty}} = 0. 
\end{equation*}
where $S(x, y, t)$ and $\mathcal{K}_{1}(x, y, t)$ are defined by \eqref{DEF-S} and \eqref{DEF-mathK}, respectively. 
\end{cor}

\noindent
\underline{\bf Step 2. Error Estimate for the Duhamel Term}

\smallskip
Moreover, we can easily prove the following formula for the third term in the right hand side of \eqref{IE-re}.  
This proposition means that $S(x, y, t)$ in the Duhamel term of \eqref{integral-eq} can be well approximated by $K(x, y, t)$, in view of the asymptotic behavior. 
\begin{prop}\label{prop.asymp-D-S-K}
Let $p\ge1$ be an integer. Assume that $u_{0}\in X^{3}(\R^{2})\cap L^{1}(\R^{2})$, $xu_{0}\in L^{1}(\R^{2})$ and $B(u_{0})$ is sufficiently small. 
Then, there exists a sufficiently large $T_{*}\ge1$ such that 
\begin{align}\label{asymp-D-S-K}
\left\| \int_{0}^{t}\p_{x}(S-K)(t-\tau)*u^{p+1}(\tau)d\tau \right\|_{L^{\infty}} 
\le CE_{0}^{p}t^{-\min\left\{\frac{9}{4}, \frac{7p+1}{4}\right\}}, \ \ t\ge T_{*}, 
\end{align}
where $S(x, y, t)$, $K(x, y, t)$ and $E_{0}$ are defined by \eqref{DEF-S}, \eqref{DEF-K} and \eqref{data-1}, respectively. 
\end{prop}
\begin{proof}
First, let us split the target integral as follows: 
\begin{align}
\int_{0}^{t}\p_{x}(S-K)(t-\tau)*u^{p+1}(\tau)d\tau
&=\left(\int_{0}^{\frac{t}{2}}+\int_{\frac{t}{2}}^{t}\right)\p_{x}(S-K)(t-\tau)*u^{p+1}(\tau)d\tau \nonumber \\
&=:J_{1}(x, y, t)+J_{2}(x, y, t). \label{S-K-split}
\end{align}
Then, analogously as \eqref{I2-est}, it follows from Corollary~\ref{cor.L-decay-LKPB-ap}, \eqref{u-sol-decay}, \eqref{u-L2L2-est} and \eqref{data-1} that 
\begin{align}
\left\|J_{1}(t)\right\|_{L^{\infty}} &\le C\int_{0}^{\frac{t}{2}}(t-\tau)^{-\frac{9}{4}}\left\|u^{p+1}(\tau)\right\|_{L^{1}}d\tau \nonumber \\
&\le C\int_{0}^{\frac{t}{2}}(t-\tau)^{-\frac{9}{4}}\left\|u(\tau)\right\|_{L^{\infty}}^{p-1}\left\|u(\tau)\right\|_{L^{2}}^{2}d\tau  \nonumber  \\
&\le CE_{0}^{p-1}\int_{0}^{\frac{t}{2}}(t-\tau)^{-\frac{9}{4}}(1+\tau)^{-\frac{7(p-1)}{4}}\left\|u(\tau)\right\|_{L^{2}}^{2}d\tau \nonumber \\
&\le CE_{0}^{p-1}\left\|\p_{x}^{-1}u_{0}\right\|_{L^{2}}^{2}t^{-\frac{9}{4}}
\le CE_{0}^{p}t^{-\frac{9}{4}}, \ \ t>0. \label{J1-est}
\end{align}

On the other hand, similarly as \eqref{I3-est}, by virtue of Lemma~\ref{lem.Duhamel-est-Linf}, \eqref{u-sol-decay} and \eqref{L2-re}, 
there exists a sufficiently large $T_{*}\ge T_{\nu} \ge1$ such that 
\begin{align}
\left\|J_{2}(t)\right\|_{L^{\infty}} 
&\le \left\| \int_{\frac{t}{2}}^{t}\p_{x}S(t-\tau)*u^{p+1}(\tau)d\tau \right\|_{L^{\infty}}+\left\|\int_{\frac{t}{2}}^{t}\p_{x}K(t-\tau)*u^{p+1}(\tau)d\tau \right\|_{L^{\infty}} \nonumber \\
&\le C\int_{\frac{t}{2}}^{t}(t-\tau)^{-\frac{3}{4}}\left(\left\|u^{p+1}(\tau)\right\|_{L^{2}}^{2}+\left\|\p_{y}\left(u^{p+1}(\tau)\right)\right\|_{L^{2}}^{2}\right)^{\frac{1}{2}}d\tau \nonumber \\
&\le C\int_{\frac{t}{2}}^{t}(t-\tau)^{-\frac{3}{4}}\left\|u(\tau)\right\|_{L^{\infty}}^{p}\left(\left\|u(\tau)\right\|_{L^{2}}^{2}+\left\|u_{y}(\tau)\right\|_{L^{2}}^{2}\right)^{\frac{1}{2}}d\tau \nonumber \\
&\le CE_{0}^{p}\int_{\frac{t}{2}}^{t}(t-\tau)^{-\frac{3}{4}}(1+\tau)^{-\frac{7p}{4}}\tau^{-\frac{1}{2}}d\tau 
\le CE_{0}^{p}t^{-\frac{7p+1}{4}}, \ \ t\ge T_{*}. \label{J2-est} 
\end{align}
Combining \eqref{S-K-split}, \eqref{J1-est} and \eqref{J2-est}, we arrive at the desired result \eqref{asymp-D-S-K}. 
\end{proof}

\noindent
\underline{\bf Step 3. Anisotropic Asymptotic Expansion on $xt$}

\smallskip
Next, we would like to perform the anisotropic asymptotic expansions in the $x$-direction and the $t$-direction, and then extract the leading term of  
\begin{equation}\label{DEF-v}
v(x, y, t):=\mathcal{K}_{1}(t)-\frac{1}{p+1}\int_{0}^{t}\p_{x}K(t-\tau)*u^{p+1}(\tau)d\tau. 
\end{equation}
To do that, let us introduce the following quantities: 
\begin{align}
&U(w, \tau):=-\frac{1}{p+1}\int_{\R}u^{p+1}(z, w, \tau)dz, \ \ w\in \R, \ \tau>0, \label{DEF-U} \\
&Q(w):=M_{1}(w)+\int_{0}^{\infty}U(w, \tau)d\tau, \ \ w\in \R, \label{DEF-Q}
\end{align}
where $M_{1}(w)$ is defined by \eqref{DEF-mathK}. Under the assumptions in Theorem~\ref{thm.main}, we can see that
\begin{equation}\label{UQ-L1}
U \in L^{1}(\R \times [0, \infty)), \ \ M_{1}\in L^{1}(\R), \ \ Q\in L^{1}(\R). 
\end{equation}
More precisely, it follows from \eqref{u-sol-decay}, \eqref{u-L2L2-est} and \eqref{data-1} that 
\begin{align}
\int_{0}^{\infty}\int_{\R}\left|U(w, \tau)\right|dwd\tau 
&\le \int_{0}^{\infty}\int_{\R^{2}}\left|u^{p+1}(z, w, \tau)\right|dzdwd\tau 
\le \int_{0}^{\infty}\left\|u(\tau)\right\|_{L^{\infty}}^{p-1}\left\|u(\tau)\right\|_{L^{2}}^{2}d\tau  \nonumber \\
&\le CE_{0}^{p-1}\int_{0}^{\infty}\left\|u(\tau)\right\|_{L^{2}}^{2}d\tau 
\le CE_{0}^{p-1}\left\|\p_{x}^{-1}u_{0}\right\|_{L^{2}}^{2}
\le CE_{0}^{p},  \label{U-est} \\
\int_{\R}\left|M_{1}(w)\right|dw &\le \int_{\R^{2}}\left|xu_{0}(x, w)\right|dxdw \le C\left\|xu_{0}\right\|_{L^{1}} \le CE_{0}. \nonumber 
\end{align}
Moreover, for the constant $\mathcal{N}_{0}$ defined by \eqref{DEF-mathN}, we note that 
\begin{equation}\label{N0-henkei}
\int_{\R}Q(w)dw=\int_{\R}M_{1}(w)dw+\int_{0}^{\infty}\int_{\R}U(w, \tau)dwd\tau=\mathcal{N}_{0}. 
\end{equation}
In addition, we shall define the new function: 
\begin{align}
V(x, y, t):=\int_{\R}\p_{x}K(x, y-w, t)Q(w)dw, \ \ (x, y)\in \R^{2}, \ t>0.  \label{DEF-V}
\end{align}
Then, the next step is to prove the following asymptotic formula: 
\begin{prop}\label{prop.aniso-x}
Let $p\ge1$ be an integer. Assume that $u_{0}\in X^{3}(\R^{2})\cap L^{1}(\R^{2})$, $xu_{0}\in L^{1}(\R^{2})$ and $B(u_{0})$ is sufficiently small. 
Then, there exists a remainder function $\Psi(x, y, t)$ satisfying 
\begin{equation}\label{aniso-x}
\left|v(x, y, t)-V(x, y, t)\right|\le CE_{0}^{p}y^{2}t^{-\frac{13}{4}}+\Psi(x, y, t), 
\end{equation}
for each $(x, y)\in \R^{2}$, $t>0$, and 
\begin{equation}\label{Psi-est}
\lim_{t\to \infty}t^{\frac{7}{4}}\left\|\Psi(t)\right\|_{L^{\infty}}=0, 
\end{equation}
where $v(x, y, t)$, $V(x, y, t)$ and $E_{0}$ are defined by \eqref{DEF-v}, \eqref{DEF-V} and \eqref{data-1}, respectively. 
\end{prop}
\begin{proof}
First, let us take a small $0<\delta<1$ and split the $\tau$-integral in \eqref{DEF-v} as follows:  
\begin{align}
&-\frac{1}{p+1}\int_{0}^{t}\p_{x}K(t-\tau)*u^{p+1}(\tau)d\tau \nonumber \\
&= -\frac{1}{p+1}\int_{\frac{\delta t}{2}}^{t}\p_{x}K(t-\tau)*u^{p+1}(\tau)d\tau \nonumber\\
&\ \ \ \, -\frac{1}{p+1}\int_{0}^{\frac{\delta t}{2}}\int_{\R^{2}} \p_{x}\left( K(x-z, y-w, t-\tau)-K(x, y-w, t-\tau) \right) u^{p+1}(z, w, \tau)dzdwd\tau \nonumber\\
&\ \ \ \, -\frac{1}{p+1}\int_{0}^{\frac{\delta t}{2}}\int_{\R^{2}} \p_{x}K(x, y-w, t-\tau) u^{p+1}(z, w, \tau) dzdwd\tau \nonumber \\
&=: R_{1}^{\delta}(x, y, t)+R_{2}^{\delta}(x, y, t)+L^{\delta}(x, y, t). \label{D_K-split-R+L}
\end{align}
In the above expression, $R_{1}^{\delta}(x, y, t)$ and $R_{2}^{\delta}(x, y, t)$ can be regarded as the remainder terms. Indeed, for $R_{1}^{\delta}(x, y, t)$, we are able to get its decay estimate analogously as \eqref{J2-est}. More precisely, there exist a certain constant $C_{\delta}>0$ and a sufficiently large $T_{\delta}\ge1$ such that 
\begin{equation}\label{Rdelta1-est}
\left\|R_{1}^{\delta}(t)\right\|_{L^{\infty}}\le C_{\delta}E_{0}^{p}t^{-\frac{7p+1}{4}}, \ \ t\ge T_{\delta}. 
\end{equation}
On the other hand, $R_{2}^{\delta}(x, y, t)$ has already been evaluated as Lemma~3.12 in \cite{FH23}. 
Since the proof is valid also under present assumptions, we omit it here. Actually, we have the same results as 
\begin{equation}\label{Rdelta2-est}
\lim_{t\to \infty}t^{\frac{7}{4}}\left\|R_{2}^{\delta}(t)\right\|_{L^{\infty}}=0. 
\end{equation}

By virtue of the above discussion, in order to prove \eqref{aniso-x}, it is sufficient to derive the leading term of $\mathcal{K}_{1}(x, y, t)+L^{\delta}(x, y, t)$. 
To do that, we further transform these integrals. Now, it follows from \eqref{DEF-mathK}, \eqref{D_K-split-R+L}, \eqref{DEF-U}, \eqref{DEF-Q} and \eqref{DEF-V} that 
\begin{align}
&\mathcal{K}_{1}(x, y, t)+L^{\delta}(x, y, t) \nonumber \\
&=\int_{\R}\p_{x}K(x, y-w, t)M_{1}(w)dw-\frac{1}{p+1}\int_{0}^{\frac{\delta t}{2}}\int_{\R^{2}}\p_{x}K(x, y-w, t-\tau)u^{p+1}(z, w, \tau)dzdwd\tau \nonumber \\
 &=\int_{\R}\p_{x}K(x, y-w, t)M_{1}(w)dw+\int_{0}^{\frac{\delta t}{2}}\int_{\R}\p_{x}K(x, y-w, t)U(w, \tau)dwd\tau \nonumber\\
 &\ \ \ +\int_{0}^{\frac{\delta t}{2}}\int_{\R}\left(\p_{x}K(x, y-w, t-\tau)-\p_{x}K(x, y-w, t)\right)U(w, \tau)dwd\tau \nonumber \\
 &=\int_{\R}\p_{x}K(x, y-w, t)Q(w)dw-\int_{\frac{\delta t}{2}}^{\infty}\int_{\R}\p_{x}K(x, y-w, t)U(w, \tau) dwd\tau \nonumber \\
&\ \ \ +\int_{0}^{\frac{\delta t}{2}}\int_{\R}\left(\p_{x}K(x, y-w, t-\tau)-\p_{x}K(x, y-w, t)\right)U(w, \tau) dwd\tau \nonumber \\
&=:V(x, y, t)+W_{1}^{\delta}(x, y, t)+W_{2}^{\delta}(x, y, t). \label{DEF-Wdelta}
\end{align}
Then, using Proposition~\ref{prop.L-decay-K} and the fact $U\in L^{1}(\R \times [0, \infty))$ comes from \eqref{UQ-L1}, we have 
\begin{align}
\limsup_{t \to \infty}t^{\frac{7}{4}}\left\|W_{1}^{\delta}(t)\right\|_{L^{\infty}} 
&\le \limsup_{t \to \infty}t^{\frac{7}{4}}\int_{\frac{\delta t}{2}}^{\infty}\int_{\R}\left\|\p_{x}K(\cdot, \cdot-w, t)\right\|_{L^{\infty}}\left|U(w, \tau)\right|dwd\tau \nonumber \\
&\le C\lim_{t \to \infty}\int_{\frac{\delta t}{2}}^{\infty}\int_{\R}\left|U(w, \tau)\right|dwd\tau=0. \label{W1-est}
\end{align}
Therefore, it suffices to analyze $W_{2}^{\delta}(x, y, t)$ only. To do that, let us decompose it as follows: 
\begin{align}
W_{2}^{\delta}(x, y, t)&=\int_{0}^{\frac{\delta t}{2}}\left(\int_{|w|\ge t^{\frac{3}{4}}}+\int_{|w|\le t^{\frac{3}{4}}}\right)\left(\p_{x}K(x, y-w, t-\tau)-\p_{x}K(x, y-w, t)\right)U(w, \tau) dwd\tau \nonumber \\
&=:W_{2.1}^{\delta}(x, y, t)+W_{2.2}^{\delta}(x, y, t).  \label{W2.1W2.2}
\end{align}

In what follows, we would like to evaluate $W_{2.1}^{\delta}(x, y, t)$ and $W_{2.2}^{\delta}(x, y, t)$. 
For $W_{2.1}^{\delta}(x, y, t)$, using Proposition~\ref{prop.L-decay-K} and the fact $U\in L^{1}(\R \times [0, \infty))$ again, it immediately follows from the Lebesgue dominated convergence theorem that 
\begin{align}
&\limsup_{t \to \infty}t^{\frac{7}{4}}\left\|W_{2.1}^{\delta}(t)\right\|_{L^{\infty}}  \nonumber \\
&\le \limsup_{t \to \infty}t^{\frac{7}{4}} \int_{0}^{\frac{\delta t}{2}}\int_{|w|\ge t^{\frac{3}{4}}}\left( \left\|\p_{x}K(\cdot, \cdot-w, t-\tau)\right\|_{L^{\infty}}+\left\|\p_{x}K(\cdot, \cdot-w, t)\right\|_{L^{\infty}}   \right)\left|U(w, \tau)\right|dwd\tau \nonumber \\
&\le C\limsup_{t \to \infty}t^{\frac{7}{4}} \int_{0}^{\frac{\delta t}{2}}\int_{|w|\ge t^{\frac{3}{4}}}\left\{ (t-\tau)^{-\frac{7}{4}}+t^{-\frac{7}{4}}\right\}\left|U(w, \tau)\right|dwd\tau \nonumber \\
&\le C\limsup_{t \to \infty}t^{\frac{7}{4}} \left\{ \left(t-\frac{\delta t}{2}\right)^{-\frac{7}{4}}+t^{-\frac{7}{4}} \right\} \int_{0}^{\frac{\delta t}{2}}\int_{|w|\ge t^{\frac{3}{4}}}\left|U(w, \tau)\right|dwd\tau \nonumber \\
&\le C\left\{\left(1-\frac{\delta}{2}\right)^{-\frac{7}{4}}+1\right\}\lim_{t\to \infty}\int_{0}^{\infty}\int_{|w|\ge t^{\frac{3}{4}}}\left|U(w, \tau)\right|dwd\tau =0. \label{est-W2.1}
\end{align}

Next, we shall treat $W_{2.2}^{\delta}(x, y, t)$. First, applying the mean value theorem to the difference in the integrand in \eqref{W2.1W2.2}, we have 
\begin{equation*}
\p_{x}K(x, y-w, t-\tau)-\p_{x}K(x, y-w, t)
=-\tau \int_{0}^{1}(\p_{t}\p_{x}K)(x, y-w, t-\theta \tau)d\theta. 
\end{equation*}
Now, recalling that $K(x, y, t)$ is the fundamental solution to \eqref{w-linear}, we obtain $\p_{t}\p_{x}K=\nu \p_{x}^{3}K-\e \p_{y}^{2}K$ for all $(x, y)\in \R^{2}$ and $t>0$. Therefore, we can see that 
\begin{align}
W_{2.2}^{\delta}(x, y, t)&=-\nu \int_{0}^{\frac{\delta t}{2}}\int_{|w|\le t^{\frac{3}{4}}}\tau\left(\int_{0}^{1}\p_{x}^{3}K(x, y-w, t-\theta \tau)d\theta\right)\left|U(w, \tau)\right|dwd\tau \nonumber \\
&\ \ \ +\e \int_{0}^{\frac{\delta t}{2}}\int_{|w|\le t^{\frac{3}{4}}}\tau\left(\int_{0}^{1}(\p_{y}^{2}K)(x, y-w, t-\theta \tau)d\theta\right)\left|U(w, \tau)\right|dwd\tau \nonumber \\
&=:W_{2.2.1}^{\delta}(x, y, t)+W_{2.2.2}^{\delta}(x, y, t). \label{W221-W222}
\end{align}
Under this expression, it follows from Proposition~\ref{prop.L-decay-K} and \eqref{U-est} that 
\begin{align}
&\limsup_{t\to \infty}t^{\frac{7}{4}}\left\|W_{2.2.1}^{\delta}(t)\right\|_{L^{\infty}} \nonumber \\
&\le C\limsup_{t\to \infty}t^{\frac{7}{4}}\int_{0}^{\frac{\delta t}{2}}\int_{|w|\le t^{\frac{3}{4}}}|\tau|\left(\int_{0}^{1}\left\|\p_{x}^{3}K(\cdot, \cdot-w, t-\theta \tau)\right\|_{L^{\infty}}d\theta\right)\left|U(w, \tau)\right|dwd\tau \nonumber \\
&\le C\delta \limsup_{t\to \infty}t^{\frac{11}{4}}\int_{0}^{\frac{\delta t}{2}}\int_{|w|\le t^{\frac{3}{4}}}\left(\int_{0}^{1}(t-\theta \tau)^{-\frac{11}{4}}d\theta\right)\left|U(w, \tau)\right|dwd\tau \nonumber \\
&\le C\delta\left(1-\frac{\delta}{2}\right)^{-\frac{11}{4}}\int_{0}^{\infty}\int_{\R} \left|U(w, \tau)\right|dwd\tau 
\le C\delta E_{0}^{p}. \label{W221-est}
\end{align}

In the rest of this proof, we would like to evaluate $W_{2.2.2}^{\delta}(x, y, t)$. To do that, let us calculate $(\p_{y}^{2}K)(x, y-w, t-\theta \tau)$. First, we have from the definition of $K(x, y, t)$, i.e., \eqref{DEF-K} that   
\[
(\p_{y}^{2}K)(x, y, t)=t^{-\frac{11}{4}}(\p_{y}^{2}K_{*})\left(xt^{-\frac{1}{2}}, yt^{-\frac{3}{4}}\right). 
\]
Then, we need to consider $\p_{y}^{2}K_{*}(x, y)$. For simplicity, we set 
\[
c_{0}:=\frac{1}{4\pi^{\frac{3}{2}}\nu^{\frac{3}{4}}}, \qquad \Theta(x, y, r):=x \sqrt{\frac{r}{\nu}} +\frac{y^{2}}{4\e}\sqrt{\frac{r}{\nu}} - \frac{\pi}{4}\e. 
\]
Under this notations, a direct calculation yields  
\begin{align*}
&K_{*}(x, y)=c_{0}\int_{0}^{\infty}r^{-\frac{1}{4}}e^{-r}\cos \Theta(x, y, r)dr, \ \ 
\p_{y}K_{*}(x, y)=-\frac{c_{0}y}{2\e\sqrt{\nu}}\int_{0}^{\infty}r^{\frac{1}{4}}e^{-r}\sin \Theta(x, y, r)dr,  \\
&\p_{y}^{2}K_{*}(x, y)=-\frac{c_{0}y^{2}}{4\nu}\int_{0}^{\infty}r^{\frac{3}{4}}e^{-r}\cos \Theta(x, y, r)dr
-\frac{c_{0}}{2\e\sqrt{\nu}}\int_{0}^{\infty}r^{\frac{1}{4}}e^{-r}\sin \Theta(x, y, r)dr. 
\end{align*}
Combining, all the above results, we arrive at 
\begin{align}
\int_{0}^{1}\left|(\p_{y}^{2}K)(x, y-w, t-\theta \tau) \right|d\theta  
&\le C(y-w)^{2}\int_{0}^{1}(t-\theta \tau)^{-\frac{17}{4}}d\theta+C\int_{0}^{1}(t-\theta \tau)^{-\frac{11}{4}}d\theta \nonumber \\
&\le C(y^{2}+w^{2})(t-\tau)^{-\frac{17}{4}}+C(t-\tau)^{-\frac{11}{4}}. \label{y2-bibun}
\end{align}
Thus, from \eqref{W221-W222}, \eqref{y2-bibun} and \eqref{U-est}, for each $(x, y) \in \R^{2}$ and $t>0$, we obtain 
\begin{align}
\left|W_{2.2.2}^{\delta}(x, y, t)\right|
&\le C\int_{0}^{\frac{\delta t}{2}}\int_{|w|\le t^{\frac{3}{4}}}|\tau|(t-\tau)^{-\frac{17}{4}}(y^{2}+w^{2})\left|U(w, \tau)\right|dwd\tau \nonumber \\
&\ \ \ +C\int_{0}^{\frac{\delta t}{2}}\int_{|w|\le t^{\frac{3}{4}}}|\tau|(t-\tau)^{-\frac{11}{4}}\left|U(w, \tau)\right|dwd\tau \nonumber \\
&\le C\delta t \left(1-\frac{\delta}{2}\right)^{-\frac{17}{4}}t^{-\frac{17}{4}}\int_{0}^{\frac{\delta t}{2}}\int_{|w|\le t^{\frac{3}{4}}}\left(y^{2}+t^{\frac{3}{2}}\right)\left|U(w, \tau)\right|dwd\tau \nonumber  \\
&\ \ \ +C\delta t \left(1-\frac{\delta}{2}\right)^{-\frac{11}{4}}t^{-\frac{11}{4}}\int_{0}^{\frac{\delta t}{2}}\int_{|w|\le t^{\frac{3}{4}}}\left|U(w, \tau)\right|dwd\tau \nonumber \\
&\le CE_{0}^{p}y^{2}t^{-\frac{13}{4}}+C\delta E_{0}^{p}t^{-\frac{7}{4}}.  \label{W222-est}
\end{align}

As a conclusion, summing up \eqref{DEF-v}, \eqref{D_K-split-R+L} through \eqref{W221-est} and \eqref{W222-est}, since $0<\delta<1$ can be chosen arbitrary small, we can eventually see that there exists a remainder function $\Psi(x, y, t)$ satisfying the desired results \eqref{aniso-x} and \eqref{Psi-est}. This completes the proof. 
\end{proof}

\noindent
\underline{\bf Step 4. Anisotropic Asymptotic Expansion on $y$}

\smallskip
Finally, in the rest of this paper, an anisotropic asymptotic expansion in the $y$-direction is performed for $V(x, y, t)$ to derive the asymptotic profile of the solution to \eqref{KPB}. The idea of the proof is based on a combination of the Dollard decomposition for the Schr\"{o}dinger equation and the asymptotic analysis for parabolic equations. Indeed, we can get the following result: 
\begin{prop}\label{prop.aniso-y}
Let $p\ge1$ be an integer. Assume that $u_{0}\in X^{3}(\R^{2})\cap L^{1}(\R^{2})$, $xu_{0}\in L^{1}(\R^{2})$ and $B(u_{0})$ is sufficiently small. 
Then, there exist a bounded function $\omega : [0, \infty)\to [0, \infty)$ satisfying $\lim_{\rho \to 0+}\omega(\rho)=0$, and a remainder function $\Phi(x, y, t)$ such that 
\begin{equation}\label{aniso-y}
\left|V(x, y, t)-\mathcal{N}_{0}\p_{x}K(x, y, t)\right|\le t^{-\frac{7}{4}}\omega\left(|y|t^{-\frac{3}{2}}\right)+\Phi(x, y, t), 
\end{equation}
for each $(x, y)\in \R^{2}$, $t>0$, and 
\begin{equation}\label{Phi-est}
\lim_{t\to \infty}t^{\frac{7}{4}}\left\|\Phi(t)\right\|_{L^{\infty}}=0, 
\end{equation}
where $V(x, y, t)$, $K(x, y, t)$ and $\mathcal{N}_{0}$ are defined by \eqref{DEF-V}, \eqref{DEF-K} and \eqref{DEF-mathN}, respectively. 
\end{prop}
\begin{proof}
First, we shall rewrite the integral kernel $K(x, y, t)$ defined in \eqref{DEF-K} as follows: 
\begin{align}
K(x, y, t)&= t^{ -\frac{5}{4} } K_{*} \left( xt^{-\frac{1}{2}}, yt^{-\frac{3}{4}} \right) \nonumber  \\
& = \frac{  t^{ -\frac{5}{4} } }{ 4\pi^{ \frac{3}{2} } \nu^{\frac{3}{4}}} \int_{0}^{\infty} r^{-\frac{1}{4} }e^{ -r } \cos \left( x \sqrt{\frac{r}{\nu t}} + \frac{y^{2}}{4\e}\sqrt{\frac{r}{\nu}}t^{-\frac{3}{2}} - \frac{\pi}{4}\e \right) dr \nonumber \\
& = \frac{  t^{ -\frac{5}{4} } }{ 8\pi^{ \frac{3}{2} } \nu^{\frac{3}{4}}} \int_{0}^{\infty} r^{-\frac{1}{4} }e^{ -r } \left( e^{ \frac{iy^{2}}{4\e}\sqrt{\frac{r}{\nu}}t^{-\frac{3}{2}} - \frac{i\pi}{4}\e +ix \sqrt{\frac{r}{\nu t}}} + e^{ -\frac{iy^{2}}{4\e}\sqrt{\frac{r}{\nu}}t^{-\frac{3}{2}} +\frac{i\pi}{4}\e -ix \sqrt{\frac{r}{\nu t}}}  \right) dr \nonumber \\
& = \frac{t^{ -\frac{1}{2} } }{ 4\pi^{ \frac{3}{2} }} \left( \int_{0}^{\infty} |\xi|^{ \frac{1}{2} } e^{ -\nu t \xi^{2}+\frac{i\xi y^{2}}{4t\e} - \frac{i\pi}{4}\e} e^{ix\xi} d\xi 
+ \int_{-\infty}^{0} |\xi|^{ \frac{1}{2} } e^{ -\nu t \xi^{2} +\frac{i\xi y^{2}}{4t\e} + \frac{i\pi}{4}\e} e^{ix\xi} d\xi \right) \nonumber \\
& = \frac{t^{ -\frac{1}{2} }}{ 4\pi^{ \frac{3}{2} }} \int_{\R} |\xi|^{ \frac{1}{2} } e^{ -\nu t \xi^{2} + \frac{i\xi y^{2}}{4t\e} - \frac{i\pi}{4}\e \mathrm{sgn} \xi } e^{ix\xi} d\xi,  \label{K-rewrite-1} 
\end{align}
and then the anisotropic Fourier transform of $\p_{x}K(x, y, t)$ on the $x$-direction can be given by 
\begin{equation}
\mathcal{F}_{x}[\p_{x}K](\xi, y, t)=\frac{t^{ -\frac{1}{2} }}{ 2^{ \frac{3}{2} }\pi} |\xi|^{ \frac{1}{2} } e^{ -\nu t \xi^{2} + \frac{i\xi y^{2}}{4t\e} - \frac{i\pi}{4}\e \mathrm{sgn} \xi }(i \xi). \label{K-rewrite-2} 
\end{equation}

Next, let us take $0< \beta<3/4$ and decompose $V(x, y, t)$ in \eqref{DEF-V} as 
\begin{align}
V(x, y, t)=\left(\int_{|w|\le t^{\beta}}+\int_{|w|\ge t^{\beta}}\right)\p_{x}K(x, y-w, t)Q(w)dw=:Z_{1}(x, y, t)+Z_{2}(x, y, t). \label{DEF-Z}
\end{align}
Then, analogously as \eqref{est-W2.1}, by virtue of $Q\in L^{1}(\R)$, which follows from \eqref{UQ-L1}, we have 
\begin{align}
\limsup_{t \to \infty}t^{\frac{7}{4}}\left\|Z_{2}(t)\right\|_{L^{\infty}}
\le C\lim_{t\to \infty}\int_{|w|\ge t^{\beta}}\left|Q(w)\right|dw=0.  \label{Z2-est}
\end{align}

In what follows, we would like to analyze $Z_{1}(x, y, t)$. It follows from \eqref{K-rewrite-1} that 
\begin{align*}
Z_{1}(x, y, t)=\frac{t^{ -\frac{1}{2} }}{ 4\pi^{ \frac{3}{2} }} \int_{\R}|\xi|^{ \frac{1}{2} } e^{ -\nu t \xi^{2} - \frac{i\pi}{4}\e \mathrm{sgn} \xi } (i\xi)e^{ix\xi} \left(\int_{|w|\le t^{\beta}}e^{ \frac{i\xi (y-w)^{2}}{4t\e} }Q(w)dw\right) d\xi. 
\end{align*}
Using the identity 
\[
e^{iz}=1+iz\int_{0}^{1}e^{i\theta z}d\theta, \ \ z\in \mathbb{R}, 
\]
the $w$-integral on the right-hand side above can be rewritten as
\begin{align*}
\int_{|w|\le t^{\beta}}e^{ \frac{i\xi (y-w)^{2}}{4t\e} }Q(w)dw
&=e^{ \frac{i\xi y^{2}}{4t\e} }\int_{|w|\le t^{\beta}}e^{ \frac{i\xi w^{2}}{4t\e} -\frac{i\xi yw}{2t\e }}Q(w)dw \nonumber \\
&=e^{ \frac{i\xi y^{2}}{4t\e} }\int_{|w|\le t^{\beta}}\left(1+\frac{i\xi w^{2}}{4t\e} \int_{0}^{1} e^{ \frac{i\theta \xi w^{2}}{4t\e} } d\theta \right)e^{-\frac{i\xi yw}{2t\e }}Q(w)dw \nonumber \\
&=e^{ \frac{i\xi y^{2}}{4t\e} }\int_{\R}Q(w)e^{-\frac{i\xi yw}{2t\e }}dw-e^{ \frac{i\xi y^{2}}{4t\e} }\int_{|w|\ge t^{\beta}}Q(w)e^{-\frac{i\xi yw}{2t\e }}dw \nonumber \\
&\ \ \ +e^{ \frac{i\xi y^{2}}{4t\e} }\int_{|w|\le t^{\beta}} \left( \int_{0}^{1} e^{ \frac{i\theta \xi w^{2}}{4t\e} } d\theta \right)\frac{i\xi w^{2}}{4t\e}e^{-\frac{i\xi yw}{2t\e }}Q(w)dw. 
\end{align*}
Therefore, we see that $Z_{1}(x, y, t)$ can be rewritten by 
\begin{align}
Z_{1}(x, y, t)&=\frac{t^{ -\frac{1}{2} }}{ 4\pi^{ \frac{3}{2} }} \int_{\R}|\xi|^{ \frac{1}{2} } e^{ -\nu t \xi^{2} +\frac{i\xi y^{2}}{4t\e} - \frac{i\pi}{4}\e \mathrm{sgn} \xi } (i\xi)e^{ix\xi} \left(\int_{\R}Q(w)e^{ -\frac{i\xi yw}{2t\e} }dw\right) d\xi \nonumber \\
&-\frac{t^{ -\frac{1}{2} }}{ 4\pi^{ \frac{3}{2} }} \int_{\R}|\xi|^{ \frac{1}{2} } e^{ -\nu t \xi^{2} +\frac{i\xi y^{2}}{4t\e} - \frac{i\pi}{4}\e \mathrm{sgn} \xi } (i\xi)e^{ix\xi} \left(\int_{|w|\ge t^{\beta}}Q(w)e^{ -\frac{i\xi yw}{2t\e} }dw\right) d\xi \nonumber \\
&+\frac{t^{ -\frac{1}{2} }}{ 4\pi^{ \frac{3}{2} }} \int_{\R}|\xi|^{ \frac{1}{2} } e^{ -\nu t \xi^{2} +\frac{i\xi y^{2}}{4t\e} - \frac{i\pi}{4}\e \mathrm{sgn} \xi } (i\xi)e^{ix\xi} \left\{\int_{|w|\le t^{\beta}}\left(\int_{0}^{1} e^{ \frac{i\theta \xi w^{2}}{4t\e} } d\theta \right) \frac{i\xi w^{2}}{4t\e}e^{-\frac{i\xi yw}{2t\e }}Q(w)dw\right\} d\xi \nonumber\\
&=:\chi_{1}(x, y, t)+\chi_{2}(x, y, t)+\chi_{3}(x, y, t).  \label{Z1}
\end{align}
Now, for $\chi_{2}(x, y, t)$ and $\chi_{3}(x, y, t)$, some direct calculations immediately yield that 
\begin{align*}
\left\|\chi_{2}(t)\right\|_{L^{\infty}}
&\le Ct^{-\frac{1}{2}}\int_{\R}|\xi|^{\frac{3}{2}}e^{-\nu t\xi^{2}}\left(\int_{|w|\ge t^{\beta}}\left|Q(w)\right|dw\right)d\xi 
\le Ct^{-\frac{7}{4}}\int_{|w|\ge t^{\beta}}\left|Q(w)\right|dw, \ \ t>0, \\
\left\|\chi_{3}(t)\right\|_{L^{\infty}}
&\le Ct^{-\frac{1}{2}}\int_{\R}|\xi|^{\frac{3}{2}}e^{-\nu t\xi^{2}}\left(\int_{|w|\le t^{\beta}}|\xi||w|^{2}t^{-1}\left|Q(w)\right|dw\right)d\xi \\
&\le Ct^{-\frac{3}{2}+2\beta}\left(\int_{\R}|\xi|^{\frac{5}{2}}e^{-\nu t\xi^{2}}d\xi\right)\left(\int_{|w|\le t^{\beta}}\left|Q(w)\right|dw\right) 
\le C\left\|Q\right\|_{L^{1}}t^{-\frac{13}{4}+2\beta}, \ \ t>0. 
\end{align*}
In addition, recalling \eqref{UQ-L1} again and noticing that $13/4-2\beta>7/4$ holds under the current situation $0<\beta<3/4$, analogously as \eqref{Z2-est}, we have 
\begin{equation}\label{chi2chi3-est}
\lim_{t \to \infty}t^{\frac{7}{4}}\left(\left\|\chi_{2}(t)\right\|_{L^{\infty}}+\left\|\chi_{3}(t)\right\|_{L^{\infty}}\right)=0. 
\end{equation}

Finally, let us deal with $\chi_{1}(x, y, t)$ and derive the leading term of this function. Now, using \eqref{Z1}, \eqref{K-rewrite-2} and some basic properties of the Fourier transform, we can see that 
 \begin{align}
 \begin{split}\label{chi-henkei}
\chi_{1}(x, y, t)&=\frac{t^{ -\frac{1}{2} }}{ 2^{ \frac{3}{2} }\pi} \int_{\R}|\xi|^{ \frac{1}{2} } e^{ -\nu t \xi^{2} + \frac{i\xi y^{2}}{4t\e} - \frac{i\pi}{4}\e \mathrm{sgn} \xi }(i \xi)\hat{Q}\left(\frac{\xi y}{2t\e}\right)e^{ix\xi}d\xi \\
&=\mathcal{F}^{-1}_{\xi}\left[\sqrt{2\pi}\mathcal{F}_{x}[\p_{x}K](\xi, y, t) \hat{Q}\left(\frac{\xi y}{2t\e}\right) \right](x) \\
&=\left(\p_{x}K(\cdot, y, t)*\mathcal{F}_{\xi}^{-1}\left[\hat{Q}\left(\frac{\xi y}{2t\e}\right)\right](\cdot)\right)(x), \ \ (x, y)\in \R^{2}, \ t>0. 
\end{split}
\end{align}
In addition, it follows from the dilation property of the Fourier transform that 
\[
\mathcal{F}_{\xi}^{-1}\left[\hat{Q}\left(\frac{\xi y}{2t\e}\right)\right](x)=\left|\frac{2t\e}{y}\right| Q\left(\frac{2t\e}{y}x\right)
\]
holds for each $y\neq0$ and $t>0$. Thus, combining these results and \eqref{N0-henkei}, we have 
\begin{align}
&\chi_{1}(x, y, t)-\mathcal{N}_{0}\p_{x}K(x, y, t) \nonumber \\
&=\int_{\R}\p_{x}K(x-z, y, t)\left|\frac{2t\e}{y}\right| Q\left(\frac{2t\e}{y}z\right)dz-\left(\int_{\R}Q(w)dw\right)\p_{x}K(x, y, t) \nonumber \\
&=\int_{\R}
\left\{\p_{x}K\left(x-\frac{yw}{2t\varepsilon}, y, t \right)-\p_{x}K(x,y,t) \right\}Q(w)dw 
\label{chi-final}
\end{align}
for $x\in \R$, $y\neq0$ and $t>0$. Moreover, if $y=0$, we note that $\sqrt{2\pi}\hat{Q}(0)=\int_{\R}Q(w)dw=\mathcal{N}_{0}$ holds since \eqref{N0-henkei}. 
Hence, we obtain from \eqref{chi-henkei} that $\chi_{1}(x, 0, t)=\mathcal{N}_{0}\p_{x}K(x, 0, t)$. 
In what follows, we shall focus on the integrand in \eqref{chi-final}. 
By the mean value theorem and Proposition~\ref{prop.L-decay-K}, we get
\begin{align}
\left|\p_{x}K\left(x-\frac{yw}{2t\varepsilon}, y, t \right)-\p_{x}K(x, y, t)\right| 
&\le \frac{|yw|}{2t}\int_{0}^{1}\left|\p_{x}^{2}K\left(x-\theta \frac{yw}{2t\varepsilon}, y, t \right)\right|\theta \nonumber \\
&\le \frac{|yw|}{2t}\cdot \left\|\p_{x}^{2}K(t)\right\|_{L^{\infty}}
\le C|w||y|t^{-\frac{13}{4}}.  \label{last-est1}
\end{align} 
On the other hand, Proposition~\ref{prop.L-decay-K} and the triangle inequality also give us that 
\begin{equation}\label{last-est2}
\left|\p_{x}K\left(x-\frac{yw}{2t\varepsilon}, y, t \right)-\p_{x}K(x, y, t)\right| \le 2\left\|\p_{x}K(t)\right\|_{L^{\infty}} \le Ct^{-\frac{7}{4}}. 
\end{equation}
Hence, summing up \eqref{chi-final}, \eqref{last-est1} and \eqref{last-est2}, there exists a positive constant $C_{0}>0$ such that 
\begin{align}
\left|\chi_{1}(x, y, t)-\mathcal{N}_{0}\p_{x}K(x, y, t)\right| 
&\le C_{0}\int_{\R}\min\left\{|w||y|t^{-\frac{13}{4}}, t^{-\frac{7}{4}}\right\}\left|Q(w)\right|dw \nonumber \\
&= t^{-\frac{7}{4}}\omega\left(|y|t^{-\frac{3}{2}}\right), \ \ (x, y)\in \R^{2}, \ t>0, \label{chi1-asymp}
\end{align}
where the above function $\omega : [0, \infty)\to [0, \infty)$ is defined by
\[
\omega(\rho):=C_{0}\int_{\mathbb R}\min\left\{\rho|w|,1\right\}\left|Q(w)\right|dw.
\]
Since $Q\in L^1(\mathbb R)$ by \eqref{UQ-L1}, the Lebesgue dominated convergence theorem yields $\lim_{\rho \to 0+}\omega(\rho)=0$. 

As a conclusion, combining \eqref{DEF-Z}, \eqref{Z2-est}, \eqref{Z1}, \eqref{chi2chi3-est} and \eqref{chi1-asymp}, we can eventually see that there exists a remainder function $\Phi(x, y, t)$ satisfying the desired results \eqref{aniso-y} and \eqref{Phi-est}. This completes the proof. 
\end{proof}

\begin{proof}[\rm{\bf{End of the Proof of Theorem~\ref{thm.main}}}]
First, the decay estimate \eqref{u-sol-decay} is a direct consequence of Proposition~\ref{prop.main-decay}. 
Moreover, summarizing up all the results of \eqref{IE-re}, Corollary~\ref{cor.asymp-linear}, Proposition~\ref{prop.asymp-D-S-K}, \eqref{DEF-v}, \eqref{DEF-V}, Propositions~\ref{prop.aniso-x} and \ref{prop.aniso-y}, we can conclude that the desired asymptotic formula \eqref{u-sol-asymp} has been established. 
Furthermore, if $|y|\leq r(t)$ with $r(t)=o(t^{3/4})$ ($t\to \infty$), then $r(t)^2t^{-3/2}\to0$ and $r(t)t^{-3/2}\to0$ as $t\to\infty$.
Therefore, \eqref{u-sol-asymp-2} follows immediately from \eqref{u-sol-asymp}, \eqref{Rr-est} and
the fact that $\omega(\rho)\to0$ as $\rho\to0+$. This completes the proof. 
\end{proof}

\section*{Acknowledgments}

\indent

This study is supported by JSPS KAKENHI Grant Number 22K13939.
The author would like to express his sincere appreciation to Professor Hiroyuki Hirayama in University of Miyazaki, for valuable comments and stimulating discussions.
The author is also grateful to Professor Shuji Yoshikawa in Hiroshima University, for pointing out an error in an earlier draft of this paper during a conference presentation. His comment was extremely helpful in revising the manuscript.



\vspace{-0.1cm}

\par\noindent
\begin{flushleft}Ikki Fukuda\\
Division of Mathematics and Physics, Faculty of Engineering, Shinshu University, \\
4-17-1, Wakasato, Nagano, 380-8553, JAPAN\\
E-mail: i\_fukuda@shinshu-u.ac.jp
\end{flushleft}

\end{document}